\RequirePackage{fix-cm}
\documentclass[smallextended]{svjour3}       
\smartqed  
\usepackage{graphicx}
\usepackage{amssymb,amsfonts,latexsym,epsfig,euscript,amsmath}
 \usepackage{color}
\usepackage{arydshln}
 \usepackage{hyperref}
\newtheorem{thm}{Theorem}[section]

\newcommand{\mc}{\multicolumn{1}{c}}
\numberwithin{equation}{section}
\numberwithin{thm}{section}

\newtheorem{dfn}[thm]{Definition}
\newtheorem{prop}[thm]{Proposition}
\newtheorem{cor}[thm]{Corollary}
\newtheorem{lem}[thm]{Lemma}
\newtheorem{algorithm}[thm]{Algorithm}

\begin{document}

\title{The Gauss quadrature for general linear functionals, Lanczos algorithm, and minimal partial realization
}

\titlerunning{Gauss quadrature for functionals, Lanczos algorithm, and partial realization}        

\author{Stefano Pozza         \and
        Miroslav Prani\'c 
}


\institute{S. Pozza \at
              Faculty of Mathematics and Physics, Charles University, Sokolovsk\'a 83, 186 75 Praha 8, Czech Republic. Associated member of ISTI-CNR, Pisa, Italy, and member of INdAM-GNCS group, Italy. \\
              \email{pozza@karlin.mff.cuni.cz}           
           \and
           M. Prani\'c \at
              Department of Mathematics and Informatics, University of Banja Luka, Faculty of Science,
 			     M. Stojanovi\'ca 2, 51000 \mbox{Banja Luka}, Bosnia and Herzegovina.
}

\date{Received: date / Accepted: date}

\maketitle

\begin{abstract}
 The concept of Gauss quadrature can be generalized to approximate linear functionals with complex moments.
  Following the existing literature, this survey will revisit such generalization. 
  It is well known that the (classical) Gauss quadrature for positive definite linear functionals
  is connected with orthogonal polynomials, and with the (Hermitian) Lanczos algorithm.
  Analogously, the Gauss quadrature for linear functionals is connected with formal orthogonal polynomials, 
  and with the non-Hermitian Lanczos algorithm with look-ahead strategy;
  moreover, it is related to the minimal partial realization problem.
  We will review these connections pointing out the relationships between several results established independently in related contexts.
  Original proofs of the Mismatch Theorem and of the Matching Moment Property are given 
  by using the properties of formal orthogonal polynomials and the Gauss quadrature for linear functionals.
\end{abstract}

\section{Introduction}\label{sec:intro}
Let $A$ be an $N \times N$ Hermitian positive definite matrix and $\mathbf{v}$ a vector so that $\mathbf{v}^*\mathbf{v}=1$,
where $\mathbf{v}^*$ is the conjugate transpose of $\mathbf{v}$.
Consider the specific linear functional $\mathcal{L}$ on the space of polynomials 
defined by 
\begin{equation}\label{eq:lin:hpd}
  \mathcal{L}(\lambda^j) := \mathbf{v}^* A^j \mathbf{v} = m_j, \quad j=0,1,\dots , 
\end{equation}
where $m_0, m_1, \dots$ are real numbers known as the \emph{moments} of $\mathcal{L}$.
The functional $\mathcal{L}$ can be expressed as the Riemann-Stieltjes integral 
with a non-decreasing positive distribution function $\mu(\lambda)$ 
supported on the real axis having finitely many points of increase; see, e.g, \cite[Section 3.5]{LieStrBook13},\cite[Section 7.1]{GolMeuBook10}, and \cite[Chapter~II, Section~3]{ChiBook78}.
For $1 \leq n \leq N$, the $n$-node (classical) Gauss quadrature approximating $\mathcal{L}$ is given by the unique $n$-node quadrature formula
which matches the first $2n$ moments, i.e., 
\begin{equation*}
  \mathcal{L}(\lambda^j) = \int_\mathbb{R} \lambda^j \, \textrm{d}\mu(\lambda) = \sum_{i=1}^{n}\omega_{i}\, (\lambda_i)^j, \quad j=0,\dots,2n-1,
\end{equation*}
with $\omega_i$ positive weights and $\lambda_i$ positive distinct nodes.
Classical results of the Gauss quadrature can be found, e.g., 
in \cite[Chapters III and XV]{SzeBook39}, \cite[Chapter I, Section 6]{ChiBook78}, \cite{Gau81}, \cite[Section 1.4]{GauBook04},
\cite[Chapter 3.2]{gautschi2011numerical}, \cite[Section 3.2]{LieStrBook13}.
The linear functional $\mathcal{L}$ can be associated with a \emph{Jacobi matrix} $J_n$
which is an $n \times n$ real symmetric tridiagonal matrix. 
For every function $f$ defined on the spectrum of $A$ and $J_n$,
the matrix $J_n$ gives an algebraic expression for the Gauss quadrature, i.e.,
\begin{equation}\label{eq:hpd:GQ}
   \mathbf{v}^* f(A) \, \mathbf{v} =  \int_\mathbb{R} f(\lambda) \, \textrm{d}\mu(\lambda) \approx \sum_{i=1}^{n}\omega_{i}\, f(\lambda_i) =  \mathbf{e}_1^T f(J_n) \, \mathbf{e}_1,
\end{equation}
where $f(A)$ and $f(J_n)$ are matrix functions, and $\mathbf{e}_1$ is the first vector of the Euclidean basis (with $\mathbf{e}_1^T$ the transpose). 
The matrix $J_n$ can be obtained by $n$ iterations of the Hermitian Lanczos algorithm with inputs $A$ and $\mathbf{v}$.
Indeed, $J_n = V_n^* A V_n$, where $V_n$ is the matrix given by the Lanczos algorithm
whose columns are an orthonormal basis of the Krylov subspace $\{\mathbf{v}, A \mathbf{v}, \dots, A^{n-1} \mathbf{v}\}$.
Hence the Hermitian Lanczos algorithm with input $A, \mathbf{v}$ 
gives a matrix formulation of the Gauss quadrature for $\mathcal{L}$.
Figure \ref{fig:schemeH} (see \cite[Figure 3.2]{LieStrBook13}) represents the connections described above. 
Such connections can be derived by the properties of \emph{orthogonal polynomials};
a detailed explanation can be found, e.g., in \cite[Chapter 3]{LieStrBook13} and \cite{GolMeuBook10}
(note that the relationships between the Conjugate Gradient method, Lanczos algorithm, and orthogonal polynomials 
were already pointed out by Hestenes and Stiefel in their seminal paper published in 1952 \cite[Sections 14--17]{HesSti52}).

  \begin{figure}[tb]
\centering
\includegraphics[width = 0.75\textwidth]{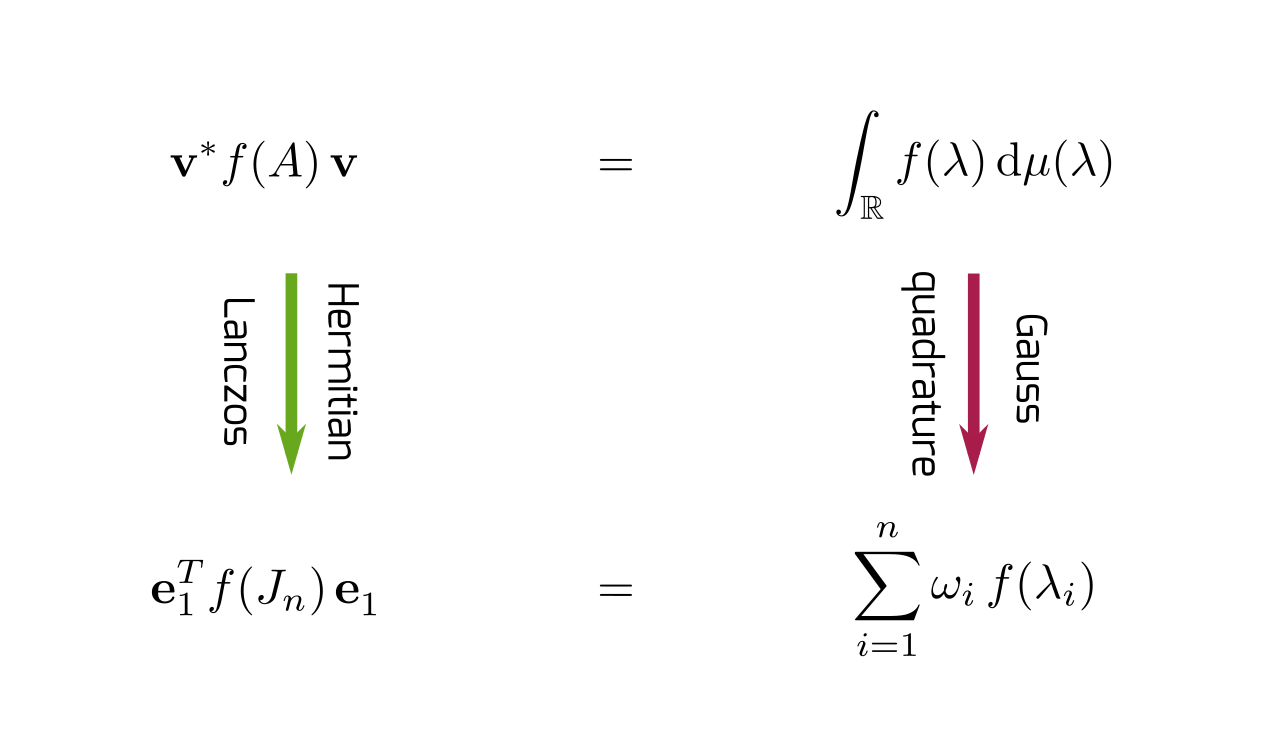}
\caption{Visualization of the connections between the (classical) Gauss quadrature and the Hermitian Lanczos algorithm.}
\label{fig:schemeH}
\end{figure}

This survey deals with the extension of the connections summarized in Figure \ref{fig:schemeH}
to the case of a general linear functional defined on the space $\mathcal{P}$ of the polynomials with generally complex coefficients,
$\mathcal{L}:\mathcal{P} \rightarrow \mathbb{C}$. 
We point out that, if not specified otherwise, 
we will consider linear functionals without the underlying assumption that they are determined by a matrix bilinear form analogous to \eqref{eq:lin:hpd}.
The survey will revisit the Gauss quadrature for linear functionals, 
its matrix formulation, its connection with the non-Hermitian Lanczos algorithm with look-ahead strategy,
and its relationship with the minimal partial realization problem. 
Furthermore, the connections between the incurable breakdown, the exactness of the Gauss quadrature, and the minimal realization problem,
will be examined with giving an original proof of the Mismatch Theorem (first proved in \cite[Theorem 4.2]{Tay82}).
The proof easily follows from the properties we will present,
providing a different interpretation of the Theorem in terms of formal orthogonal polynomials roots and nodes of the Gauss quadrature for linear functionals.

Information about the topics mentioned above and their mutual relationships are scattered in the literature. 
The survey aims to describe such topics and their connections
from the point of view of \emph{formal orthogonal polynomials}.
We hope that such a presentation will be of interest for readers working in related different areas.

Regarding the formal orthogonal polynomials and the Gauss quadrature generalization, we will mainly follow the book \cite{DraBook83} by Draux
where the Gauss quadrature is extended for the approximation of real-valued linear functionals. 
More precisely, a straightforward extension of Draux's definition to the case of complex-valued linear functionals will be presented.
The more recent Gauss quadrature definitions in \cite{Mil03} and in \cite{PozPraStr16,PozPraStr18}, obtained independently of \cite{DraBook83}, 
can be seen as a generalization to the complex quasi-definite case.
Indeed, for a real quasi-definite linear functional the quadratures in \cite{DraBook83,Mil03,PozPraStr16,PozPraStr18} are equivalent.
However, some results in \cite{PozPraStr16,PozPraStr18} do not have a counterpart in the real setting of \cite{DraBook83}
(for instance, formal orthonormal polynomials may have complex coefficients).
The case of a quasi-definite linear functional is simpler to treat; see, e.g., \cite{ChiBook78,PozPraStr16,PozPraStr18}.  
The survey will first recall the primary results associated with quasi-definite functionals and then deal with the case of a general linear functional.

This survey approaches the Lanczos algorithm in a finite dimensional setting.
Hence we will not treat infinite dimensional problems. 
For infinite dimensional problems related to positive definite linear functionals
refer, e.g., to \cite[Chapter II, Section 3, in particular Theorem 3.1]{ChiBook78}.
For the relationship with infinite dimensional Krylov subspace methods refer, e.g., to \cite{VorBook65}, \cite{GunHerSac14},
and \cite[Chapter 5]{malek2015} where many references to original works can be found.

Throughout the survey, we will consider only computations in exact arithmetic.
Since rounding errors substantially affect computations with short recurrences,
the results described in this survey cannot be applied to finite precision computations without a thorough analysis. 
Such analysis is out of the scope of this survey.
The interested reader can refer to \cite{Bai94} and \cite{Day93,Day97} for analysis of the non-Hermitian Lanczos algorithm in finite precision (assuming no breakdown);
see also the related works \cite{BaiDayYe99,TonYe00,PaiPanZem14}.
As pointed out in \cite[Sections 2.5.6 and 5.11]{LieStrBook13}, 
in finite precision arithmetic the short recurrences cannot preserve the biorthogonality or even the linear independence of the computed Krylov subspace basis.
Therefore look-ahead techniques for the non-Hermitian Lanczos have a limited impact in computing sufficiently well-conditioned basis when dealing with the loss of biorthogonality.
The interplay of look-ahead techniques and rounding errors in practical computations is still an open issue.

The paper is organized as follows.
Section \ref{sec:qdef} summarizes basic results of quasi-definite linear functionals. 
Section \ref{sec:FOP} recalls properties of formal orthogonal polynomials and of quasi-orthogonal polynomials 
with respect to a linear functional $\mathcal{L}:\mathcal{P}\rightarrow \mathbb{C}$.
The concept of Gauss quadrature for linear functionals and its matrix interpretation can be found respectively in Section \ref{sec:GQ}
and Section \ref{sec:mmp}.
The Gauss quadrature connections with the minimal partial realization problem 
and with the look-ahead Lanczos algorithm are described respectively in Section \ref{sec:minreal} and Section \ref{sec:lanczos}.
Section \ref{sec:conc} concludes the survey summarizing the links between the Gauss quadrature, minimal partial realization, and look-ahead Lanczos algorithm.

\section{Quasi-definite linear functionals}\label{sec:qdef}
We start recalling several results for quasi-definite linear functionals following the description in our previous works with Zden\v ek Strako\v s \cite{PozPraStr16,PozPraStr18}. These results will be extended to the more challenging general case in the remaining sections.

Let $\mathcal{L}: \mathcal{P} \rightarrow \mathbb{C}$ be a linear functional 
with complex moments,
\begin{equation}\label{def moments}
\mathcal{L}(\lambda^{k})=m_{k},\quad k=0,1,\ldots\,.
\end{equation}
An $n$-degree polynomial $p_n(\lambda) \in \mathcal{P}$ is called \emph{formal orthogonal polynomial (FOP)}
when it satisfies the \emph{orthogonality conditions} with respect to $\mathcal{L}$
\begin{equation*}
    \mathcal{L}(p_n\,\lambda^j)=0,  \quad \textrm{ for } j=0,\dots,n-1;
\end{equation*} 
 refer, e.g., to \cite[Introduction and Section 1.1]{DraBook83} and \cite[Chapter 2]{Bre02}.
 Notice that in \cite{BreBook80} $p_n(\lambda)$ is referred as \emph{general orthogonal polynomial};
 cf. the concept of \emph{weak orthogonal polynomial} in \cite[definition on p.~137]{Krall1966} and \cite[Section 2]{KwoLit97}.
 The subindex $n$ in the polynomial notation $p_n(\lambda)$ will always stand for the degree of the polynomial
 and we will not emphasize it further on.
 Moreover, whenever appropriate the argument $\lambda$ will be skipped for simplicity of notation.

Denoting with $\mathcal{P}_k \subset \mathcal{P}$ the subspace of polynomials of degree at most $k$,
the following classes of linear functionals can be defined (see, e.g., Theorem 3.1, Definition 3.2, Theorem 3.4 and the subsequent Corollary in \cite[Chapter I]{ChiBook78}, Theorem 1 and the subsequent Remark in \cite[Chapter VII]{LorWaaBook92}).
\begin{dfn}\label{def:quasidef}
The linear functional $\mathcal{L}$ is said to be \emph{quasi-definite} on $\mathcal{P}_{k}$
if there exist unique FOPs $p_0(\lambda),\dots,p_k(\lambda)$
(we always use the term ``unique'' for a polynomial in the sense of unique up to multiplication by a nonzero scalar)
satisfying the conditions
\begin{equation*}
 \mathcal{L}(p_j p_n) = 0  , \textrm{ for } j \neq n, 
   \textrm{ and } \, \mathcal{L}(p_n^2)\neq 0.
\end{equation*}
A linear functional is said to be \emph{positive definite} on $\mathcal{P}_k$ if in addition $p_0(\lambda),\dots,p_k(\lambda)$ are real polynomials, and $\mathcal{L}(p_n^2) > 0$, for $n=0,\dots,k$.
\end{dfn}
Note that the definition above is equivalent to the ones in \cite[Definition 2.1 and Definition 3.1]{PozPraStr16} and \cite[Definition 1.1]{PozPraStr18}.

If a FOP $p_n$ is such that $\mathcal{L}(p_n^2) = 1$, then it is a formal \emph{orthonormal} polynomial.
A beautiful summary about FOPs in the quasi-definite case can be found in the book by Chihara \cite{ChiBook78}
(notice that Chihara used the simplified term orthogonal polynomials instead of formal orthogonal polynomials).
A sequence of formal orthonormal polynomials $p_0(\lambda), \dots, p_k(\lambda)$ satisfy the three-term recurrence
\begin{equation}\label{eq:3term:ast1}
\beta_{n}  p_{n}(\lambda) = (\lambda - \alpha_{n-1}) p_{n-1} (\lambda) - \beta_{n-1} p_{n-2}(\lambda), \quad \quad n=1, \dots, k,
\end{equation} 
where $\beta_0=0$, $p_{-1}(\lambda) = 0$, $p_0(\lambda) = {1}/{\sqrt{m_0}}$
and the coefficients $\alpha_{n-1}$, $\beta_n$ are given by
\begin{equation*}
    \alpha_{n-1} = \mathcal{L}(\lambda p_{n-1} p_{n-1}),\,\,\,
    \beta_{n} =  \mathcal{L}(\lambda p_{n-1} p_{n}).
\end{equation*}
see, e.g., \cite[Chapter I, Section 4]{ChiBook78}, \cite[Theorem 2.4]{BreBook80}. 
Notice that in order to avoid ambiguity, we always take the principal value of the complex square root, i.e., we consider $arg(\sqrt{c}) \in (-\pi/2,\pi/2]$.
The recurrences \eqref{eq:3term:ast1} can be written in the compact form as 
\begin{equation*}
  \lambda \, \mathbf{p}(\lambda) = J_n \mathbf{p}(\lambda) + \beta_n p_n(\lambda)  \mathbf{e}_n, \quad n = 1, \dots, k,
\end{equation*}
where $\mathbf{p}(\lambda) = [p_0(\lambda), p_1(\lambda), \dots, p_{n-1}(\lambda)]^T$,
$\mathbf{e}_n$ is the $n$th vector of the Euclidean basis,
and $J_n$ is the $n$th complex Jacobi matrix
 \begin{equation}\label{eq:jacobi}
    J_{n} = \left [ \begin{matrix}
                    \alpha_0 & \beta_1   &              &             \\
                    \beta_1 & \alpha_1  &  \ddots      &             \\
                    \vspace{4pt}  &  \ddots   &  \ddots      & \beta_{n-1}  \\
                             &          & \beta_{n-1} & \alpha_{n-1}
                 \end{matrix}
         \right ], \quad n = 1, \dots, k;
 \end{equation}
 more information about complex Jacobi matrices and their properties can be found, e.g., in \cite{Bec01} and in \cite[in particular Section 4]{PozPraStr16}.

Given a smooth enough function $f(\lambda)$, the Gauss quadrature for quasi-definite linear functionals considered in \cite{PozPraStr16,PozPraStr18} has the form
\begin{equation*}
\mathcal{G}_n(f) :=
\sum_{i=1}^{\ell}\sum_{j=0}^{s_i-1}\omega_{i,j}\,f^{(j)}(\lambda_i),\quad
n=s_1+\,\cdots\,+s_{\ell},
\end{equation*}
and satisfies the following properties.
\begin{itemize}
  \item G1: the quadrature $\mathcal{G}_n(f)$ has maximal degree of exactness $2n-1$, 
	    i.e., it is exact for all polynomials of degree at most $2n-1$;
  \item G2: the quadrature $\mathcal{G}_n(f)$ is well-defined and it is unique.
  Moreover, Gauss quadratures with a smaller number of weights also exist
  and they are unique;
  \item G3: the quadrature $\mathcal{G}_n(f)$ can be written as the matrix form
$m_0\,\mathbf{e}_1^Tf(J_n) \mathbf{e}_1$, where $J_n$ is the
complex Jacobi matrix associated with $\mathcal{L}$. 
\end{itemize} 
A quadrature having properties G1, G2 and G3 exists 
if and only if the linear functional $\mathcal{L}$ is quasi-definite on $\mathcal{P}_n$;
see \cite[Section 7, in particular Corollaries 7.4 and 7.5]{PozPraStr16} and \cite[Theorem 3.1]{PozPraStr18}.
\\

Property G3 corresponds to the so called \emph{Matching Moment Property} of the complex Jacobi matrix, i.e.,
if the complex numbers $m_0, \dots, m_{2n-1}$ define a quasi-definite linear functional \eqref{def moments} with associated Jacobi matrix $J_n$
(here and in the following the simplified term \emph{quasi-definite linear functional} and \emph{positive definite linear functional}
 will stand for linear functionals that are quasi-definite and positive definite on the space of polynomials of sufficiently large degree), then
\begin{equation}\label{eq:mmp:qdef}
 m_0\,\mathbf{e}_1^T \, (J_n)^j \mathbf{e}_1 = m_j, \quad j=0,\dots, 2n-1; 
\end{equation}
see \cite[Section 5]{PozPraStr16}.
In \cite[Theorem 2]{FreHoc93} the Matching Moment Property was proved for a quasi-definite linear functional given by
\begin{equation*}
\mathcal{L}(f) = \mathbf{w}^* f(A) \mathbf{v},
\end{equation*}
where $A$ is a complex matrix and $\mathbf{w}, \mathbf{v}$ are vectors (compare also with \cite[Theorem 1]{Cyb87}).
In \cite{Str09} it was derived by the Vorobyev method of moments  (see in particular Chapter III of \cite{VorBook65}).

\section{Polynomials and orthogonality}\label{sec:FOP}
Let $\mathcal{L}: \mathcal{P} \rightarrow \mathbb{C}$ be a linear functional  
with moments $m_0, m_1, \dots$.
Consider 
the sequence of $k$-dimensional Hankel matrices
\begin{equation}\label{eq:hankel:subm}
 H_{k-1}=H(1:k;1:k)=\left[\,\, \begin{matrix}
\vspace{3pt} m_0 & m_1&\ldots & m_{k-1}\\
m_1 & m_2&\ldots & m_{k}\\
\vdots& \vdots&  &\vdots \\
m_{k-1}&m_{k} &\ldots& m_{2k-2}
\end{matrix} \,\,  \right], \quad k=1,2,\ldots\,,
\end{equation}
with the corresponding determinant $\Delta_{k-1}$
(the notation $A(i:j;\ell:n)$
stands for the submatrix of $A$ composed of the elements in the rows from $i$ to $j$
and in the columns from $\ell$ to $n$).
Setting ${\bf m}_{k-1}$ as the vector
$${\bf m}_{k-1}=H(1:k;k+1)=[m_{k},\ldots,m_{2k-1}]^T,$$
we are interested in the properties of the linear system
\begin{equation}\label{moments system n}
H_{k-1}\,{\bf c}=-{\bf m}_{k-1}.
\end{equation}
The solution of Hankel systems, and many related properties of Hankel matrices,
have been extensively treated in the literature;
see, e.g, to the seminal paper by Stieltjes \cite[Sections 8--11, pp.~624--630]{Sti1894}
(please notice that we refer to the English translation published by Springer in 1993), 
the monographs \cite[Chapter I]{ChiBook78}, \cite[Chapter 1]{DraBook83}, \cite{Ioh82}, \cite[Part I]{HeiRos84}, and \cite[Chapter 2]{BulVBaBook97}, and the paper \cite[Section 2]{GraLin83}.
Here, we refer in particular to some results in Section 1.2 of \cite{DraBook83};
their straightforward generalization to the complex case is equivalent to Theorems \ref{theorem key result} and \ref{theorem existence phi} given in this section;
see also Theorem 7 in \cite[Chapter XV, \S 10]{GanBook59} in the context of infinite Hankel matrices with finite rank.
We do not report the proofs of Theorems \ref{theorem key result} and \ref{theorem existence phi} since they are based on the study of Hankel matrices,
and they would lead us too far from the main point of the survey.
We will use them as the starting point of our presentation.
\begin{thm}\label{theorem key result}
Assume that $\Delta_{k-1}\neq 0$, then $\Delta_{k}=0$ if and only if
\begin{equation}\label{moments linear combination 0}
-m_{2k}=c_0m_{k}+c_1m_{k+1}+\ldots+c_{k-1}m_{2k-1},
\end{equation}
where ${\bf c}=[c_0,\ldots,c_{k-1}]^T$ is the unique solution of the
linear system \eqref{moments system n}.
Moreover, if $\Delta_{k-1}\neq 0$ and  $\Delta_k=\Delta_{k+1}=\ldots=\Delta_{k + j -1}=0$  for $j\geq1$,
then $\Delta_{k+j}=0$ if and only if
\begin{equation}\label{moments linear combination}
-m_{2k+j}=c_0m_{k+j}+c_1m_{k+j+1}+\ldots+c_{k-1}m_{2k+j-1}.
\end{equation}
\end{thm}
\noindent As a consequence, we get the following theorem; see \cite[Property 1.6]{DraBook83}.
\begin{thm}\label{theorem existence phi}
Assume that $\Delta_{k-1}\neq 0$ and
$\Delta_k=\Delta_{k+1}=\ldots=\Delta_{k+j-1}=0$. Then the system
\begin{equation}\label{moments system n+k-1}
H_{k+j-1}\,{\bf b}=-{\bf m}_{k+j-1}\end{equation} has (infinitely many) solutions 
if and only if $\Delta_{k+j}=\Delta_{k+j+1}=\ldots\Delta_{k+2j-1}=0$.
\end{thm}
\noindent The following theorem gives necessary and sufficient conditions for the existence (and uniqueness) of a FOP $p_n(\lambda)$ of degree $n$; see \cite[Property 1.14]{DraBook83}. 
\begin{thm}\label{thm:exist:wop}
Let $\mathcal{L}: \mathcal{P} \rightarrow \mathbb{C}$ be a linear functional.
An $n$-degree monic FOP exists if and only if one of the following conditions is satisfied.
\begin{itemize}
 \item $\Delta_{n-1} \neq 0$ (unique monic FOP);
 \item $\Delta_{k-1}\neq 0$ and
$\Delta_k=\Delta_{k+1}=\ldots=\Delta_{n-1}=\ldots = \Delta_{2n-k-1}=0$
(infinitely many monic FOPs);
\end{itemize}
where $\Delta_0,\Delta_1,\dots$ are the determinants of the Hankel submatrices $H_0, H_1, \dots $ composed of the moments of $\mathcal{L}$.
\end{thm}
\begin{proof}
A monic  FOP of degree $n$
$$\pi_n(\lambda)=\lambda^n+c_{n-1}\lambda^{n-1}+\ldots+c_1\lambda+c_0$$
exists if and only if 
$\mathcal{L}(\lambda^j\pi_n)=0$, for $j=0,\ldots,n-1$,
which gives the linear system \eqref{moments system n} with $k=n$.
Therefore if $\Delta_{n-1}\neq 0$, then the polynomial $\pi_n(\lambda)$ exists and is unique. 
If $\Delta_{n-1}=0$, then necessary and sufficient conditions for the existence of $\pi_{n}(\lambda)$ 
are given by Theorem \ref{theorem existence phi}: for $\Delta_{k-1}\neq 0$ and
$\Delta_k=\Delta_{k+1}=\ldots=\Delta_{n-1}=0$, there exist infinitely many $\pi_n(\lambda)$ if and only if $\Delta_{n}=\Delta_{n+1}=\ldots\Delta_{2n-k-1}=0$. 
\qed
\end{proof}

\noindent Note that by Theorem \ref{thm:exist:wop}, a linear functional $\mathcal{L}$ is quasi-definite on $\mathcal{P}_k$ if and only if $\Delta_j \neq 0$, for $j=0,1,\dots,k$; see, e.g., \cite[Chapter  I, Theorem 3.1]{ChiBook78}.

The second item of Theorem \ref{thm:exist:wop} can be interpreted in the following way: 
consider the sequence $\Delta_0, \Delta_1, \Delta_2, \dots$. 
Let $R=R(n-1)$ be the number of zeros in the sequence between $\Delta_{n-1}$ 
and the first nonzero element in the sequence after $\Delta_{n-1}$, i.e., 
$\Delta_{n-1+j}=0$ for $j=1,\ldots,R$ and $\Delta_{n+R}\neq 0$.
Note that the parameters $R(n-1)$ are known as Kronecker index, and the differences $R(n)-R(n-1)$ as Euclidean indices;
see \cite{BulVBaBook97,Kai80}.
Let $L=L(n-1)$ be the number of zeros in the sequence between $\Delta_{n-1}$ 
and the last nonzero element in the sequence before $\Delta_{n-1}$, i.e., 
$\Delta_{n-1-j}=0$ for $j=1,\ldots,L$ and $\Delta_{n-L-2}\neq 0$. 
A FOP of degree $n$ exists if and only if $R(n-1)>L(n-1)$.
Roughly said, there are ``more consecutive zeros to the right than to the left''.

Among the formal orthogonal polynomials the following cases can be distinguished;
see Definition on p.~47 of \cite{DraBook83}. 
\begin{dfn}
A formal orthogonal polynomial (FOP) $p_n(\lambda)$ is called \emph{regular} when $\Delta_{n-1}\neq 0$ (i.e., when it is unique),
while it is called \emph{singular} when $\Delta_{n-1} = 0$ (i.e., when it is not unique).
\end{dfn}

\begin{prop}\label{prop:orthreg}
   Let $p_k$ be a regular FOP and $j\ge 0$. Then $\mathcal{L}(p_k q) = 0$ for every $q \in \mathcal{P}_{k+j}$ if and only if $\Delta_k=\Delta_{k+1}=\ldots=\Delta_{k + j}=0$.
\end{prop}
\begin{proof}
 Without loss of generality, let us assume $p_k$ to be monic.
 The conditions $\mathcal{L}(p_k \, \lambda^{k+i}) = 0$, for $i=0,\dots,j$, lead to the system
 $$ -m_{2k+i}=c_0m_{k+i}+c_1m_{k+i+1}+\ldots+c_{k-1}m_{2k+i-1}, \quad i=0,\dots,j, $$
 with $c_0, \dots, c_{k-1}$ the unique solution of the
linear system \eqref{moments system n}.
 Theorem \ref{theorem key result} concludes the proof.
\end{proof}

Knowing all the integers $k$ such that $\Delta_k=0$
allows determining all the integers $n$ for which a FOP $p_n(\lambda)$ exists.

{\it Example 1.} If the zero-nonzero pattern of the sequence of Hankel determinants $\Delta_k$ is
$$
\begin{array}{cccccccccccccccc}
   \Delta_k & =  & \ast & \ast & 0 & \ast & \ast & 0 & 0 & 0 & 0 & \ast & 0 & 0 & 0 & \ast \\
   k        & =  & 0 & 1 & 2 & 3 & 4 & 5 & 6 & 7 & 8 & 9 & 10 & 11 & 12 & 13  
\end{array},
$$
then the FOPs of degree $3, 8, 9, 12$ and $13$ do not exist. 
There exist regular FOPs of degree $1, 2, 4, 5, 10$ and $14$ and
singular FOPs of degree $6, 7$ and $11$.

In order to fill the gaps in FOP sequences, we consider polynomials satisfying the following property.
\begin{dfn}\label{def:qop}
The polynomial $p_n(\lambda)$ is called \emph{quasi-orthogonal of order $k$} (or \emph{$k$-quasi-orthogonal}),
with $k<n$, when
$$\mathcal{L}(p_n\lambda^j)=0,\quad j=0,\ldots,n-k-1.$$
\end{dfn}

\vspace{-10pt}
\noindent Quasi-orthogonal polynomials of order $1$ were introduced by Riesz in \cite{Rie23} and then generalized to any order by Chihara in \cite{Chi57};
see also \cite[Definition 1.1, p.~51]{DraBook83}, \cite{Dra90}, \cite{Dra16},
and compare the definition with the concept of \emph{inner formal orthogonal polynomials} given in \cite[Definition 5.2]{hochbruck:1996}
and of \emph{left and right quasi-formally biorthogonal polynomials} in \cite[Definition 3.3]{Fre93b}.
Note that Definition \ref{def:qop} does not require $k$ to be minimal, i.e., it is not necessary that $\mathcal{L}(p_n\lambda^{n-k})\ne 0$. Thus a $k$-quasi-orthogonal polynomial of degree $n$ is also $j$-quasi-orthogonal for $j=k+1,\dots,n-1$. Also, any formal orthogonal polynomial of degree $n$ is $k$-quasi-orthogonal for $k=1,\dots,n-1$.

If $\Delta_{k-1} \neq 0$, then an $(n-k)$-quasi orthogonal polynomial of degree $n$ exists for every $n$ larger than $k$; 
see, e.g., \cite[Lemma 3.4]{Fre93b}. 
The following theorem will prove it together with the characterization of such polynomials; 
see discussion on pp.~47--51 of \cite{DraBook83}. 
\begin{thm}\label{theorem rr for mqop} 
Let $\Delta_0,\Delta_1,\dots$ be the Hankel determinants associated with the linear functional $\mathcal{L}$. 
Let $\Delta_{k-1}\neq 0$, and $\Delta_{k-1+i}=0$ for $i=1,\ldots,j$, 
and let $\pi_k(\lambda)$ be the regular monic FOP with respect to $\mathcal{L}$. 
Then all the monic $i$-quasi-orthogonal polynomials $\pi_{k+i}(\lambda)$ for $i=1,\dots,j$ are of the form
\begin{equation}\label{shape for monic qop}
\pi_{k+i}(\lambda)=\pi_k(\lambda)\prod_{t=1}^{i}(\lambda-\eta_t),  \quad  \eta_t \in \mathbb{C}.
\end{equation}
\end{thm}
\begin{proof}
The proof is by induction on $i$.
Let $i=1$. By Proposition \ref{prop:orthreg}, 
$\Delta_k=0$ if and only if $\pi_k(\lambda)$ is orthogonal to all polynomials of degree $k$.
Therefore $\lambda\pi_k(\lambda)$ is a monic polynomial of degree $k+1$ that is orthogonal to $\mathcal{P}_{k-1}$:
$$\mathcal{L}(\lambda\,\pi_k\,q)=\mathcal{L}(\pi_k\,(\lambda\,q))=0,\quad \textrm{ for } q(\lambda) \in \mathcal{P}_{k-1}.$$
Moreover, any polynomial of the form $(\lambda-\alpha)\pi_k(\lambda)$, $\alpha\in\mathbb{C}$,
is a monic $1$-quasi-orthogonal polynomial. 
On the other side, assume that $p_{k+1}(\lambda)$ is an arbitrary monic polynomial of
degree $k+1$ that is orthogonal to $\mathcal{P}_{k-1}$. 
Then the polynomial $\lambda\pi_k(\lambda) - p_{k+1}(\lambda)$ has the following two properties:
\begin{itemize}
  \item it is of degree $k$, 
  \item it is orthogonal to $\mathcal{P}_{k-1}$.
\end{itemize} 
Hence the uniqueness of $\pi_k(\lambda)$ gives $\lambda\pi_k(\lambda) - p_{k+1}(\lambda)=\beta \pi_k(\lambda)$ for a certain complex number $\beta$, i.e.,
$$p_{k+1}(\lambda)=(\lambda-\beta)\,\pi_k(\lambda),\quad \textrm{ for some } \beta \in\mathbb{C}.$$

Set $i$ between $2$ and $j-1$, 
and assume that all the monic $i$-quasi-orthogonal polynomials of degree $k+i$ are of the form \eqref{shape for monic qop}. 
By Proposition \ref{prop:orthreg},
$\Delta_k=\Delta_{k+1}=\ldots=\Delta_{k+i}=0$ 
if and only if $\pi_k(\lambda)$ is orthogonal to all polynomials of degree $k+i$. 
Therefore $\lambda\pi_{k+i}(\lambda)$ is a monic polynomial of degree $k+i+1$ that is orthogonal to $\mathcal{P}_{k-1}$:
$$\mathcal{L}(\lambda\,\pi_{k+i}\,q)=\mathcal{L}(\pi_k\,(\lambda\,(\lambda-\eta_1)\cdots(\lambda-\eta_i)\,q))=0,\quad \textrm{ for } q(\lambda) \in \mathcal{P}_{k-1}.$$ 
Clearly, $(\lambda-\alpha)\pi_{k+i}(\lambda)$ is a monic $(i+1)$-quasi-orthogonal polynomial of degree $k+i+1$, 
for any complex number $\alpha$. 
It remains to prove that an arbitrary monic polynomial of degree $k+i+1$ that is orthogonal to $\mathcal{P}_{k-1}$
is of the form $(\lambda-\beta)\pi_{k+i}(\lambda)$, where $\beta$ is a certain complex number, 
and $\pi_{k+i}(\lambda)$ is a polynomial of the form (\ref{shape for monic qop}). 
It can be done similarly to the case $i=1$.
\qed
\end{proof}

\begin{prop}
Let $\Delta_0,\Delta_1,\dots$ be the Hankel determinants associated with the linear functional $\mathcal{L}$ such that 
$\Delta_{k-1}\neq 0$ and $\Delta_{k+i}=0$ for $i=0,\dots,2j-1$. 
Then for $i=0,\dots,j$, $p_{k+i}(\lambda)$ is a FOP
if and only if it is $i$-quasi-orthogonal.
\end{prop}
\begin{proof}
 Clearly any FOP of degree $k+i$ is $i$-quasi-orthogonal.
 Vice versa if $p_{k+i}(\lambda)$ is $i$-quasi-orthogonal, 
 then it satisfies \eqref{shape for monic qop}.
 By Proposition~\ref{prop:orthreg}, $p_k(\lambda)$ is orthogonal to $\mathcal{P}_{k+2j-1}$. Therefore if $q(\lambda) \in \mathcal{P}_{k+i-1}$,
 then  $\mathcal{L} (p_{k+i} q) = \mathcal{L} (p_{k} (\lambda - \eta_1) \cdots (\lambda - \eta_i) q) = 0$ for $i=0,\dots,j$.
 \qed
\end{proof}

Consider the sequence of polynomials
\begin{equation}\label{eq:seq:p}
   p_0(\lambda), p_1(\lambda), p_2(\lambda), \dots 
\end{equation}
 constructed in the following way: 
$p_n(\lambda)$ is a regular FOP (when possible) or $p_n(\lambda)$ is a $(n-k)$-quasi-orthogonal polynomial, 
where $p_k(\lambda)$ is the last regular FOP before $p_n(\lambda)$. 
For later convenience, we consider every nonzero choice for $p_0(\lambda)$ as a regular FOP.
Let us denote by $\nu{(0)},\nu{(1)},\nu{(2)},\dots$ all the indexes for which $p_{\nu{(j)}}(\lambda)$ is a regular FOP,
i.e., $\Delta_{\nu{(j)-1}} \neq 0$ (setting $p_{\nu(0)}(\lambda) = p_0(\lambda) \neq 0$, and $\nu(j+1) = \infty$ when $p_{\nu(j)}(\lambda)$ is the last of the regular FOPs).
By Theorem \ref{theorem rr for mqop},
the quasi-orthogonal polynomials between two consecutive regular FOPs $p_{\nu(j)}(\lambda), p_{\nu(j+1)}(\lambda)$
satisfy the recurrences
\begin{equation}\label{eq:gen:rec:qorth}
    \beta_{n}p_{n}(\lambda) = \lambda p_{n-1}(\lambda) - \sum_{i=\nu(j)}^{n-1} \alpha_{n,i} p_i(\lambda), \quad n=\nu(j)+1,\dots,\nu(j+1)-1,
\end{equation}
for some coefficients $\alpha_{n,i} \in \mathbb{C}$ and $\beta_n \neq 0$;
see \cite[Theorem 1.5 and Remark 1.2]{DraBook83}.
Notice that any choice of $\alpha_{n,i}$ and $\beta_n \neq 0$ defines a $(n-\nu(j))$-quasi-orthogonal polynomial.
In particular, there exist families of such polynomials satisfying the two-term recurrences
\begin{equation}\label{two term rr}
    \beta_n p_{n}(\lambda)=(\lambda - \alpha_{n,n-1}) p_{n-1}(\lambda), \quad n=\nu(j)+1,\dots,\nu(j+1)-1;
\end{equation}
fixing $\alpha_{n,n-1}=0$ gives even simpler recurrences.

Setting $n = \nu{(j+1)}$ for some $j \geq 0$,
the regular FOP $p_n(\lambda)$ satisfies (see \cite[Theorem 1.5 and Remark 1.2]{DraBook83} and \cite[Theorem 2]{GraLin83}) 
\begin{equation}\label{long rr by Freund}
 \beta_n p_{n}(\lambda) = \lambda p_{n - 1}(\lambda) - \sum_{i=\nu(j)}^{n-1}\alpha_{n,i} p_i(\lambda) - \gamma_{n}p_{\nu(j-1)}(\lambda),
\end{equation}
with $p_{\nu(-1)}(\lambda) = p_{-1}(\lambda) = 0$, $\beta_n$ a nonzero coefficient, $\gamma_{\nu(1)} = 0$,
$$ \gamma_{n} = \frac{ \mathcal{L}( \lambda p_{n-1} p_{\nu(j)-1}) }{ \mathcal{L}(p_{\nu(j)-1}  p_{\nu(j-1)}) } \neq 0, \quad j \geq 1, $$
and $\alpha_{n,i}$ given by
\begin{equation}
\left[\begin{array}{ccc}\label{system for alpha}
\mathcal{L}( p_{\nu(j)}   p_{\nu(j)} ) &  \ldots & \mathcal{L}( p_{\nu(j)}   p_{n-1}) \\
\vdots &     & \vdots \\
\mathcal{L}( p_{n-1} p_{\nu(j)}) &\ldots &\mathcal{L}( p_{n-1} p_{n-1})
\end{array}  \right]
\left[\begin{array}{c} 
	\alpha_{n,\nu(j)} \\
	\vdots \\
	\alpha_{n,n-1}
\end{array}\right]
=
\left[\begin{array}{c} \mathcal{L}( \lambda p_{\nu(j)} p_{n-1} )\\  
                       \vdots \\
                       \mathcal{L}( \lambda p_{n-1} p_{n-1}) 
\end{array}
\right],
\end{equation}
where the matrix of the system is nonsingular; see, e.g., \cite[Theorem 2.3]{Fre93:num:math}. 
Notice that for a quasi-definite linear functional
the related (regular) formal orthonormal polynomials satisfy the three term recurrences \eqref{eq:3term:ast1}.

Given $n = 1, 2, \dots$, the recurrences \eqref{eq:gen:rec:qorth} and \eqref{long rr by Freund} 
can be expressed in the matrix form 
(see \cite[Section 1.7]{DraBook83}, \cite[Section 3]{pinar:ramirez}; c.f.,
\cite[pp.~221--222]{Gra74}, \cite[Figure 2 and Theorem 3]{GraLin83}, and \cite[Equalities (3.4) and (3.5)]{FreGutNac93})
\begin{equation}\label{eq:matrix:rec:relation}
\lambda \, \mathbf{p}(\lambda)
  =
  T_{n} \, \mathbf{p}(\lambda)
  +
  \beta_{n} p_n(\lambda) \mathbf{e}_n,
\end{equation}
 with $\mathbf{p}(\lambda) = [p_0(\lambda), p_1(\lambda), \dots, p_{n-1}(\lambda)]^T$
 and where $T_{n}$ is the block matrix 
  \begin{equation}\label{block:Tn}
 \renewcommand{\arraystretch}{1.2}
     T_{n} = \left [ 
     \,\,
    \begin{array}{ c c | c c | c c | c c}
    \cline{1-2}
    \multicolumn{1}{|c}{} & &  & \mc{ } &  & \mc{ } & & \\
    \multicolumn{2}{|c|}{\raisebox{.6\normalbaselineskip}[0pt][0pt]{$\quad A_0 \quad$}} & \beta_{\nu{(1)}}  & \mc{ } &   & \mc{ }  &  &  \\
    \cline{1-4}
      &   & & &   & \mc{ } & & \\
    \gamma_{\nu{(2)}} &   & \multicolumn{2}{c|}{\raisebox{.6\normalbaselineskip}[0pt][0pt]{$\quad A_1 \quad$}} & \multicolumn{1}{l}{\ddots} & \mc{ } & & \\
    \cline{3-6}
      & \mc{ } &   &   & & \multicolumn{1}{c|}{} & & \\
      & \mc{ } & \multicolumn{1}{l}{\ddots} &   & \multicolumn{2}{c|}{\raisebox{.6\normalbaselineskip}[0pt][0pt]{$\quad \ddots \quad$}} & \beta_{\nu{(\ell)}} & \\
    \cline{5-8}
      & \mc{ } &   & \mc{ } &              &   & & \multicolumn{1}{c|}{}   \\
      & \mc{ } &   & \mc{ } & \gamma_{n} &   & \multicolumn{2}{c|}{\raisebox{.6\normalbaselineskip}[0pt][0pt]{$\quad A_\ell^{n} \quad$}} \\
     \cline{7-8}
  \end{array}
  \,\,
      \right],
\end{equation}
with the coefficients $\beta_{\nu{(j)}}$ on the first upper diagonal, 
the coefficients $\gamma_{\nu{(j)}}$ in the position $(\nu{(j)},\nu{({j-2)}}+1)$, 
$\gamma_n = 0$ when $p_n(\lambda)$ is not regular, and 
$$ A_j^n = \left [
\begin{array}{ccccc}
  \alpha_{\nu{(j)}+1, \nu{(j)}}   & \beta_{\nu{(j)} +1}               & 0               & \dots  & 0  \\
  \alpha_{\nu{(j)}+2, \nu{(j)}}   & \alpha_{\nu{(j)} +2, \nu{(j)} + 1}     & \beta_{\nu{(j)} +2}  & \ddots & \vdots  \\
  \vdots                & \vdots                       &  \ddots         & \ddots & 0   \\
  \alpha_{n-1, \nu{(j)}}     & \alpha_{n-1, \nu{(j)} +1}         &  \dots          & \ddots & \beta_{n-1}  \\
  \alpha_{n, \nu{(j)}}       & \alpha_{n, \nu{(j)} +1}           &  \dots          & \dots            & \alpha_{n, n-1}  
\end{array}
\right],
$$
for $n=\nu{(j)}+1,\dots,\nu{({j+1)}}$; $A_j = A_j^{\nu{({j+1)}}}$ for simplicity.
Notice that using the recurrences \eqref{two term rr} with $\alpha_{n,n-1} = 0$ 
gives the sparse matrix
$$ A_j =  \left [  
\begin{array}{ccccc}
  0       & \beta_{\nu{(j)} +1}      & 0               	& \dots  & 0  \\
  0       & 0             	& \beta_{\nu{(j)} +2}  	& \ddots & \vdots  \\
  \vdots  & \vdots              &  \ddots         	& \ddots & 0   \\
  0       & 0             	&  \dots          	& 0      & \beta_{\nu{({j+1)}}-1}  \\
  \alpha_{\nu{({j+1)}}, \nu{(j)}} & \alpha_{\nu{({j+1)}}, \nu{(j)} +1}    &  \dots 	& \dots  & \alpha_{\nu{({j+1)}}, \nu{({j+1)}}-1}  
\end{array}
\right],
$$
with $\alpha_{\nu{({j+1)}}, \nu{(j)}}, \dots, \alpha_{\nu{({j+1)}}, \nu{({j+1)}}-1}$ obtained by \eqref{system for alpha}.

When the polynomials $p_0(\lambda),\dots, p_n(\lambda)$ are regular FOPs (the linear functional is quasi-definite on $\mathcal{P}_{n-1}$) 
the blocks $A_j$ are scalars.
Therefore $T_n$ is an irreducible tridiagonal matrix 
since $\beta_j$ and $\gamma_{j+1}$ are nonzero for $j=1,\dots,n-1$.
In particular, there exists a sequence of formal orthonormal polynomials
so that the matrix $T_n$ is the complex Jacobi matrix \eqref{eq:jacobi}.

\section{The Gauss quadrature for linear functionals}\label{sec:GQ}
Given a linear functional $\mathcal{L}$ and a smooth enough function $f(\lambda)$, 
consider a quadrature approximating $\mathcal{L}(f)$ of the form
(see \cite[Chapter 5]{DraBook83}, \cite[Section 2]{Mil03}, and \cite[Section 7]{PozPraStr16})
\begin{equation}\label{GQ multiple nodes}
\mathcal{G}_n(f) :=
\sum_{i=1}^{\ell}\sum_{j=0}^{s_i-1}\omega_{i,j}\,f^{(j)}(\lambda_i),\quad
n=s_1+\,\cdots\,+s_{\ell},
\end{equation}
with $\omega_{i,j}$ the weights, $\lambda_i$ the \emph{distinct} nodes, and $s_i$
the multiplicity of the node $\lambda_i$. 
Notice that the number of nodes $\ell$ can be less than $n$. 
The quadrature \eqref{GQ multiple nodes} will be referred as \emph{$n$-node quadrature}
when $\omega_{i, s_i-1} \neq 0$ for $i=1, \dots, \ell$.
Otherwise, the sum of the multiplicities would be smaller than $n$.
For any choice of (distinct) nodes $\lambda_1, \dots, \lambda_\ell$ and their multiplicities $s_i$,
such that $s_1+\,\cdots\,+s_{\ell}=n$, it is possible to achieve
that the quadrature (\ref{GQ multiple nodes}) is exact for any $f(\lambda) \in \mathcal{P}_{n-1}$. 
It is necessary and sufficient to set the weights as
\begin{equation}\label{weights}
   \omega_{i,j} = \mathcal{L}(h_{i,j}),
\end{equation}
where $h_{i,j}(\lambda)$ are polynomials from $\mathcal{P}_{n-1}$ such that
\begin{equation}\label{eq:hij}
\begin{split}
  h_{i,j}^{(t)}(\lambda_k) &= 1 \quad \textrm{ for } \lambda_k= \lambda_i \textrm{ and } t=j, \\
  h_{i,j}^{(t)}(\lambda_k) &= 0 \quad \textrm{ for } \lambda_k\neq \lambda_i \textrm{ or } t\neq j,
\end{split}
\end{equation}
with $k=1,2,\dots,\ell$, and $t=0,1,\dots,s_i-1$;
see \cite[Theorem 5.1]{DraBook83} or the proof of Theorem 7.1 in \cite{PozPraStr16}.
In this case (\ref{GQ multiple nodes}) is known as \emph{interpolatory} quadrature,
since it can be given by applying $\mathcal{L}$ to the generalized (Hermite) interpolating polynomial for the function $f(\lambda)$ at the nodes
$\lambda_i$ of the multiplicities $s_i$.
An interpolatory quadrature is completely determined by its nodes and multiplicities.
Therefore in the following a quadrature $\mathcal{G}_n$ will be said to be determined by a polynomial $p_n(\lambda)$
when it is an interpolatory quadrature \eqref{GQ multiple nodes}
with $\lambda_i$ being the roots of $p_n$, and $s_i$ the corresponding multiplicities of the roots.

The following definition is a straightforward extension to the complex case 
of the Gauss quadrature introduced by Draux in \cite[Chapter 5]{DraBook83}.
\begin{dfn}\label{def:WGQ}
The quadrature (\ref{GQ multiple nodes}) is called the \emph{$n$-node Gauss quadrature} when it is exact on the space $\mathcal{P}_{2n-1}$
and $\omega_{i,s_i-1} \neq 0$ for $i=1,\dots,\ell$ (the number of nodes, counting the multiplicities, is $n$).
\end{dfn}

\noindent We point out the following remarks:
\begin{itemize}
\item the algebraic degree of exactness of the $n$-node Gauss quadrature is allowed to be larger than $2n-1$;
\item a Gauss quadrature with smaller number of nodes may or may not exist when the $n$-node Gauss quadrature exists.
\end{itemize}
Hence the $n$-node Gauss quadrature generally does not satisfy properties G1--G3 in Section \ref{sec:qdef}.
However, when $\mathcal{L}$ is a quasi-definite linear functional, 
then the $n$-node Gauss quadrature for $\mathcal{L}$ satisfies properties G1--G3,
i.e., in this case Definition \ref{def:WGQ} is equivalent to the one in \cite{PozPraStr16,PozPraStr18}.

In order to give conditions for the existence of an $n$-node Gauss quadrature for a linear functional the following result is needed;
see \cite[Theorem 5.2]{DraBook83}, see also \cite[Theorem 1.45]{GauBook04} for positive definite linear functionals
and \cite[Theorem 7.1]{PozPraStr16} for quasi-definite linear functionals.

\begin{thm} \label{theorem about ADE} 
A quadrature $\mathcal{G}_n$ determined by a polynomial $p_n(\lambda)$ is exact for all the polynomials in $\mathcal{P}_{n + k - 1}$
if and only if $p_n(\lambda)$ is $(n-k)$-quasi-orthogonal.
\end{thm}
\begin{proof}
Assume $\mathcal{G}_n$ to be exact for every polynomial in $\mathcal{P}_{n + k - 1}$. 
Then $p_n(\lambda)$ is $(n-k)$-quasi-orthogonal. Indeed, 
$$ \mathcal{L}(p_n q) = \mathcal{G}_n(p_n q) = \sum_{i=1}^{\ell}\sum_{j=0}^{s_i-1}\omega_{i,j}\,(p_n q)^{(j)}(\lambda_i) = 0, \quad q(\lambda) \in \mathcal{P}_{k-1}, $$
since $p_n^{(j)}(\lambda_i) = 0$ for $j=0, \dots, s_i -1$, $i=1, \dots, \ell$.
Inversely, let $p_n(\lambda)$ be $(n-k)$-quasi-orthogonal.
Any $f(\lambda) \in \mathcal{P}_{n+k-1}$ can be written as $f(\lambda) = p_n(\lambda) q(\lambda) + r(\lambda)$ 
for some $q(\lambda) \in \mathcal{P}_{k-1}$  and $r(\lambda) \in \mathcal{P}_{n-1}$,
giving $\mathcal{L}(f) = \mathcal{L}(r)$.
Since $\mathcal{G}_{n}$ is interpolatory it is exact on $\mathcal{P}_{n-1}$ and thus 
$$ \mathcal{L}(f) = \mathcal{L}(r) = \sum_{i=1}^{\ell}\sum_{j=0}^{s_i-1}\omega_{i,j}\,r^{(j)}(\lambda_i). $$
The proof is concluded since $f^{(j)}(\lambda_i) = r^{(j)}(\lambda_i)$ for $j=0,\dots, s_i-1$, $i=1,\dots,\ell$.
\qed
\end{proof}
\noindent Note that the proof is a straightforward adaptation of the classical well-known argument used for proving the same result in the positive definite case.

As discussed in Section \ref{sec:FOP}, for every linear functional $\mathcal{L}$
there exists a sequence of polynomials $p_0(\lambda), p_1(\lambda), \dots$ \eqref{eq:seq:p}
so that $p_n(\lambda)$ is a regular FOP (when possible),
or $p_n(\lambda)$ is $(n-k)$-quasi-orthogonal, where $p_k(\lambda)$ is the last regular FOP before $p_n(\lambda)$
($p_0(\lambda) \neq 0$ is assumed to be regular). 
We denote by $\nu(0)=0, \nu(1), \dots$ the indexes of the regular FOPs 
(with $\nu(t+1) = +\infty$ when $\nu(t)$ is the last of the regular FOPs).
Theorem \ref{theorem about ADE} implies the following corollary (see \cite[Theorem 5.2]{DraBook83}).
\begin{cor}
   Let $p_n(\lambda)$ be a polynomial in the sequence described above. 
   \begin{itemize}
    \item If $p_n(\lambda)$ is a regular FOP,
	    then it determines a quadrature $\mathcal{G}_n$ exact for every polynomials in $\mathcal{P}_{2n-1}$.
\item if $p_n(\lambda)$ is a $(n-k)$-quasi-orthogonal polynomial,
	    then it determines a quadrature $\mathcal{G}_n$ exact for every polynomials in $\mathcal{P}_{n+k-1}$.
   \end{itemize}
\end{cor}

\noindent 
Notice that if $\nu{(1)} > 1$, then $\Delta_0 = \dots = \Delta_{\nu(1)-2} = 0$. Thus $m_j=0$ for $j=0,\dots,\nu{(1)} -2$ (see, e.g., \cite[Property 1.15]{DraBook83}) and, consequently, $\mathcal{G}_n(f) \equiv 0$ for $n=1, \dots, \nu(1) -1$.

If $\omega_{i, s_i-1} = 0$ for some $i$, then the quadrature \eqref{GQ multiple nodes} 
has a smaller number of nodes (counting the multiplicities).
The following lemmas deal with this issue; see \cite[Theorem 5.3]{DraBook83}.

\begin{lem}\label{lemma:Gnreg}
   Consider the quadratures $\mathcal{G}_n$ determined by the polynomial $p_n(\lambda)$ in the sequence described above. 
   Given two consecutive regular FOPs $p_{\nu(t)}(\lambda)$ and $p_{\nu(t+1)}(\lambda)$, with $t\geq 1$, then
   $$ \mathcal{G}_n = \mathcal{G}_{\nu(t)}, \quad \textrm{ for } n=\nu(t) + 1, \dots, \nu(t+1) -1.$$   
\end{lem}
\begin{proof}
    Theorem \ref{theorem rr for mqop}  gives
    $$ p_n(\lambda) = p_{\nu(t)}(\lambda) q_{n - \nu(t)}(\lambda), $$   
    for some polynomial $q_{n - \nu(\lambda)}(\lambda)$.  
    Let $\lambda_1, \dots, \lambda_{\ell}$ be the roots of $p_n(\lambda)$ 
    with multiplicities $s_1, \dots, s_\ell$.
    The weights of the quadrature $\mathcal{G}_n$ are given by \eqref{weights}.
    Consider the pair $i,j$ so that $(\lambda - \lambda_i)^j$ is not a factor of $p_{\nu(t)}(\lambda)$,
    i.e., the root $\lambda_i$ is not a root of $p_{\nu(t)}(\lambda)$ or it is a root of $p_{\nu(t)}(\lambda)$
    but with $j$ greater than the multiplicity of $\lambda_i$ as a root of $p_{\nu(t)}(\lambda)$.
    Then the $(n-1)$-degree interpolatory polynomial $h_{i,j}(\lambda)$ defined in \eqref{eq:hij} is a multiple of $p_{\nu(t)}(\lambda)$, i.e.,
    $$ h_{i,j}(\lambda) = p_{\nu(t)}(\lambda) r_{n - \nu(t) - 1}(\lambda), $$
     for some polynomial $r_{n - \nu(t) - 1}(\lambda)$.
     By Proposition \ref{prop:orthreg}, $p_{\nu(t)}(\lambda)$ is orthogonal to $\mathcal{P}_{\nu(t+1) -2}$,
     giving
     $$ \omega_{i,j} = \mathcal{L}(h_{i,j}) = \mathcal{L} (p_{\nu(t)} r_{n - \nu(t) - 1}) = 0. $$
    Therefore $\mathcal{G}_n$ has at most $\nu(t)$ nodes.
    Moreover, each node of $\mathcal{G}_n$ is a node of $\mathcal{G}_{\nu(t)}$
    and has multiplicity smaller than or equal to the one of the corresponding node of $\mathcal{G}_{\nu(t)}$.
     
    If $\lambda_i$ is a root of $p_{\nu(t)}(\lambda)$ with multiplicity $j$,
    then there exists a polynomial $\widetilde{h}_{i,j}(\lambda)$ of the kind of \eqref{eq:hij} 
    so that $\widetilde{\omega}_{i,j} = \mathcal{L}(\widetilde{h}_{i,j})$ is the corresponding weight of $\mathcal{G}_{\nu(t)}$.
    Since $\widetilde{h}_{i,j}(\lambda)$ has degree $\nu(t) - 1$ the weight  $\widetilde{\omega}_{i,j}$ is given by
    $$ \widetilde{\omega}_{i,j} = \mathcal{G}_n (\widetilde{h}_{i,j}). $$
    Noticing that $\mathcal{G}_n (\widetilde{h}_{i,j}) = \omega_{i,j}$  concludes the proof. 
    \qed
\end{proof}

\begin{lem}\label{lemma:nonzeroweight}
   If $p_n(\lambda)$ is a regular FOP, with $n \geq 1$, then it determines a quadrature \eqref{GQ multiple nodes} 
   such that $\omega_{i, s_i -1} \neq 0$, for $i=1,\dots,\ell$.
 \end{lem}
 \begin{proof}
   Let $t$ be such that $p_{\nu(t)}(\lambda) = p_n(\lambda)$ and $h_{i,j}(\lambda)$ as in \eqref{eq:hij}, 
   then 
   \begin{eqnarray*}
    \mathcal{L}(h_{i, s_i-1} p_{\nu(t-1)}) &=& \sum_{r=1}^\ell \sum_{s=0}^{s_i -1} \omega_{r,s} (h_{i, s_i-1} p_{\nu(t-1)})^{(s)} (\lambda_r) \\
		    &=& \sum_{r=1}^\ell \sum_{s=0}^{s_i -1} \omega_{r,s} \sum_{u=0}^{s} {s \choose u} h_{i, s_i-1}^{(u)}(\lambda_r) p_{\nu(t-1)}^{(s-u)} (\lambda_r) \\   
		    &=& \sum_{s=0}^{s_i -1} \omega_{i,s} \sum_{u=0}^{s} {s \choose u} h_{i, s_i-1}^{(u)}(\lambda_i) p_{\nu(t-1)}^{(s-u)} (\lambda_i) \\
	            &=& \omega_{i,s_i-1} p_{\nu(t-1)} (\lambda_i).	            
   \end{eqnarray*}
  Proposition \ref{prop:orthreg} gives $\mathcal{L}(h_{i, s_i-1} p_{\nu(t-1)}) \neq 0$, concluding the proof.
  \qed
 \end{proof}

\noindent The following theorem summarizes the previous discussion; see \cite[Theorems 5.2 and 5.3]{DraBook83}.

\begin{thm}\label{ade of wgq}
The $n$-node Gauss quadrature $\mathcal{G}_n$ exists (and is unique) if and only if $\Delta_{n-1}\neq 0$.
Moreover, if $\Delta_{n} = \Delta_{n+1} = \dots = \Delta_{n+j} = 0$,
then $\mathcal{G}_n$ has degree of exactness at least $2n + j$.
In particular, if $n = \nu(t)$, then $\mathcal{G}_n$ has (maximal) degree of exactness $\nu(t) + \nu(t+1) -2$,
with $\nu(t+1) = +\infty$ when $n$ is the last of the regular FOPs.
\end{thm}
\begin{proof}
   By Theorem \ref{theorem about ADE}, $\mathcal{G}_n$ is exact on $\mathcal{P}_{2n-1}$ 
   if and only if it is determined by a FOP with degree $n$, i.e., a polynomial $p_n(\lambda)$ orthogonal to $\mathcal{P}_{n-1}$.
   By Lemma \ref{lemma:Gnreg} if $p_n(\lambda)$ is a singular FOP,
   then $\mathcal{G}_n$ has not $n$ nodes. 
   Therefore it is not a $n$-node Gauss quadrature.
   Considering Lemma \ref{lemma:nonzeroweight} and noticing that regular FOPs are unique,
   $\mathcal{G}_n$ exists and is unique if and only if $\Delta_{n-1}\neq 0$.
   The proof is conclude noticing that Theorem \ref{theorem about ADE} and Lemma \ref{lemma:Gnreg} imply 
   that $\mathcal{G}_{n}$ is exact on $\mathcal{P}_{2n+j}$.
   \qed
\end{proof}

\section{Matrix formulation of the Gauss quadrature}\label{sec:mmp}
If $\mathcal{L}$ is a quasi-definite linear functional,
then the associated complex Jacobi matrix \eqref{eq:jacobi} satisfies the Matching Moment Property \eqref{eq:mmp:qdef}.
We will give an original proof of an extension of the Matching Moment Property for a general sequence of moments 
using the properties of the formal orthogonal polynomials and of the Gauss quadrature for the linear functionals.
The presented extension also considers the case of moments so that $m_0 = \dots = m_{\nu(1)} = 0$.
The case of a linear functional of the kind $\mathcal{L}(f) = \mathbf{w}^* f(A) \mathbf{v}$,  with $m_0 \neq 0$, was treated in \cite[Theorem 2.10]{GuoRen04}. 
We remark that assuming real moments (with a straightforward extension to the complex case),
the Matching Moment Property presented here, as well as the ones in \cite{FreHoc93,GuoRen04,PozPraStr16}, 
can be derived by Theorem 5 of the 1983 paper by Gragg and Lindquist \cite{GraLin83},
where such property is related to the minimal partial realization problem.

Let  $\mathcal{L}: \mathcal{P} \rightarrow \mathbb{C}$ be a linear functional
and let $T_n$ be the corresponding block tridiagonal matrix \eqref{block:Tn} associated with the sequence of polynomials $p_0(\lambda),\dots,p_n(\lambda)$. 
Denote by $p_{\nu{(t)}}(\lambda)$ the subsequence of the regular FOPs
and recall that for $\nu{(t)} < n < \nu{({t+1)}}$ the polynomials $p_n(\lambda)$ are $(n-\nu{(t)})$-quasi-orthogonal.
Also recall that if $\nu{(1)} \geq 2$, then $m_j=0$ for $j=0,\dots,\nu{(1)} -2$. 
Since the elements in the superdiagonal of $T_n$ are nonzero
the block tridiagonal matrix $T_n$ is nonderogatory, i.e., its eigenvalues have geometric multiplicity $1$.
Indeed, if $\lambda$ is an eigenvalue, 
then deleting the first column and the last row of $T_n - \lambda I$ 
gives a lower triangular nonsingular matrix (with $I = [\mathbf{e}_1, \dots, \mathbf{e}_n]$ the identity matrix). 
Thus the null space of $T_n - \lambda I$ has dimension $1$. 
Proving the Matching Moment Property will need the following lemmas.

\begin{lem}\label{thm:char:poly}
   Let $T_n$ and $p_n(\lambda)$ be as in \eqref{eq:matrix:rec:relation}. Then $p_n(\lambda)$ is the characteristic polynomial of $T_n$ (up to a nonzero rescaling).
\end{lem}

\noindent Lemma \ref{thm:char:poly} is a consequence of Lemma 2 in \cite{kautsky:81};
see also \cite[Theorem 1.11]{DraBook83}.

\begin{lem}\label{lemma:n-1}
 Let $T_1, T_2, \dots$ be a sequence of block tridiagonal matrices \eqref{block:Tn}.
 For $n \geq \nu{(1)} + 1$ the matrices $T_{n-1}$ and $T_n$ satisfy
 $$ \mathbf{e}_1^T \, (T_{n-1})^k \, \mathbf{e}_{\nu{(1)}} = \mathbf{e}_1^T \, (T_{n})^k \, \mathbf{e}_{\nu{(1)}}, \quad \text{ for }  k=0,\dots,n-1,$$
where the vectors $\mathbf{e}_1, \mathbf{e}_{\nu(1)}$ have dimension $n-1$ on the left-hand side
and $n$ on the right-hand side (we use the same notation for the sake of simplicity).
\end{lem}
\begin{proof}
Consider the $n$-dimensional vectors 
$$ \mathbf{u}_k = (T_n)^k \, \mathbf{e}_{\nu(1)}, \quad k=0, 1, \dots . $$
If the last element of $\mathbf{u}_k$ is zero for $k=0, \dots, n-1$, 
then 
$$  \mathbf{e}_1^T \mathbf{u}_k = \mathbf{e}_1^T \, (T_{n-1})^k \, \mathbf{e}_{\nu{(1)}}, \quad k=0, \dots, n-1, $$
proving the lemma.
In the following, when the elements from the position $i$ to the position $j$ of a vector are possibly nonzero,
we denote them by $*_{i:j}$ ($*_i = *_{i:i}$).
Similarly, when the elements from the position $i$ to the position $j$ are null, we denote them by $0_{i:j}$.
Direct computations show that
$$ \mathbf{u}_1 = \left[\begin{matrix}
                      0_{1:\nu(1)-2} \\
                      *_{\nu(1)-1:\nu(1)} \\
                      0_{\nu(1)+1:n}
                  \end{matrix}\right], \,
   \mathbf{u}_2 = \left[\begin{matrix}
                      0_{1:\nu(1)-3} \\
                      *_{\nu(1)-2:\nu(1)} \\
                      0_{\nu(1)+1:n}
                  \end{matrix}\right], \dots, \,
   \mathbf{u}_{\nu(1)-1} = \left[\begin{matrix}
                      *_{1:\nu(1)} \\
                      0_{\nu(1)+1:n} \\
                  \end{matrix}\right].
 $$
 Moreover,
 $$ \mathbf{u}_{\nu(1)} = \left[\begin{matrix}
                      *_{1:\nu(1)} \\
                      0_{\nu(1)+1:\nu(2)-1} \\
                      *_{\nu(2)}\\
                      0_{\nu(2)+1:n}
                  \end{matrix}\right], 
$$ 
and 
$$
   \mathbf{u}_{\nu(1)+1} = \left[\begin{matrix}
                      *_{1:\nu(1)} \\
                      0_{\nu(1)+1:\nu(2)-2} \\
                      *_{\nu(2)-1:\nu(2)} \\
                      0_{\nu(2)+1:n}
                  \end{matrix}\right], \dots,
   \mathbf{u}_{\nu(1)+\nu(2)-1} = \left[\begin{matrix}
                      *_{1:\nu(2)} \\
                      0_{\nu(2)+1:n}
                  \end{matrix}\right].
 $$
 Repeating the argument gives
 $$ \mathbf{u}_{n-1} = \left[\begin{matrix}
                      *_{1:n-1} \\
                      0
                  \end{matrix}\right], $$
 concluding the proof. 
 \qed
\end{proof}

\begin{thm}[{Matching Moment Property}]\label{matching moments main}
  Let $\mathcal{L}$ be a linear functional with complex moments $m_0, m_1, \dots$,
  and let $T_n$ be the associated block tridiagonal matrix \eqref{block:Tn}
  with the corresponding polynomials $p_0(\lambda),\dots,p_n(\lambda)$.
  Denote the indexes of the regular FOPs by $\nu{(0)}=0, \nu(1), \nu(2), \dots$ .
  For every $n \geq \nu(1)$ let $t$ be so that $\nu{(t)} \leq n < \nu(t+1)$, the matrix $T_n$ satisfies
 \begin{equation*}\label{eq:match:moment1}
       \mu \, m_{\nu{(1)}-1} \, \mathbf{e}_1^T (T_n)^k \, \mathbf{e}_{\nu{(1)}} = m_k, \,\,\, k=0,\dots, \nu(t) + \nu(t+1) -2,
 \end{equation*}
 with 
 $\mu = (\beta_1 \cdots \beta_{\nu{({1})}-1})^{-1}$ for $\nu{(1)} > 1$, $\mu = 1$ for $\nu{(1)} = 1$,
 and $\nu(t+1) = +\infty$ when $p_{\nu(t)}$ is the last regular FOP.
\end{thm}

\begin{proof}
Consider the linear functional
$$ \mathcal{L}^{(n)}(f) =  \mu \, m_{\nu{(1)}-1} \, \mathbf{e}_1^T f(T_n) \, \mathbf{e}_{\nu{(1)}}, \quad f(\lambda) \in \mathcal{P}. $$
If the linear functionals $\mathcal{L}$ and $\mathcal{L}^{(n)}$ are identical on the space $\mathcal{P}_{\nu(t) + \nu(t+1) -2}$, then the proof is given.
By Lemma \ref{thm:char:poly} and the Cayley--Hamilton Theorem, the polynomial $p_n(\lambda)$ satisfies the orthogonality conditions 
\begin{equation}\label{eq:cayley}
 \mathcal{L}^{(n)}(\lambda^k p_n) =  \mu \, m_{\nu{(1)}-1} \, \mathbf{e}_1^T (T_n)^k p_n(T_n) \, \mathbf{e}_{\nu{(1)}} = 0, \quad k=0,1,\dots \, .
\end{equation}

Proceeding by induction on $n$, first consider the case $n = \nu{(1)}> 1$. 
Since $T_{\nu(1)}$ is a Hessenberg matrix it satisfies
$$\mathbf{e}_1^T (T_{\nu(1)})^k \, \mathbf{e}_{\nu{(1)}} = 0 = m_k, \quad k=0,\dots,\nu{(1)}-2.$$
Direct computations give
$\mathbf{e}_1^T (T_{\nu(1)})^{\nu(1) - 1} \, \mathbf{e}_{\nu{(1)}} = \beta_1 \cdots \beta_{\nu(1)-1} \neq 0$.
Therefore
\begin{equation}\label{eq:nu1:mmp}
  \mu \, m_{\nu{(1)} - 1} \, \mathbf{e}_1^T (T_{\nu(1)})^{k} \, \mathbf{e}_{\nu{(1)}} = m_{k}, \quad k=0,\dots,\nu{(1)} -1, 
\end{equation}
which also trivially stands for $n = \nu(1) = 1$.
Using property \eqref{eq:cayley} and Theorem \ref{ade of wgq}, 
$p_{\nu(1)}(\lambda)$ determines the quadrature $\mathcal{G}_{\nu(1)}^{(\nu(1))}$ for $\mathcal{L}^{(\nu(1))}$ 
so that $\mathcal{G}_{\nu(1)}^{(\nu(1))}(f) = \mathcal{L}^{(\nu(1))}(f)$ for every $f(\lambda) \in \mathcal{P}$.
Moreover, $p_{\nu(1)}(\lambda)$ determines the Gauss quadrature $\mathcal{G}_{\nu(1)}$ for $\mathcal{L}$,
exact for polynomials of degree at most $\nu(1) + \nu(2) -2$.
The two quadratures $\mathcal{G}_{\nu(1)}^{(\nu(1))}$ and $\mathcal{G}_{\nu(1)}$ coincide since they have the same weights.
Indeed, if $h_{i,j}(\lambda)$ is the interpolatory polynomial \eqref{eq:hij} for $n=\nu(1)$,
then the weights of $\mathcal{G}_{\nu(1)}^{(\nu(1))}$ and $\mathcal{G}_{\nu(1)}$ are respectively given by
$$ {\omega}_{i,j}^{(\nu(1))} = \mathcal{L}^{(\nu(1))}(h_{i,j}) \quad \textrm{ and } \quad \omega_{i,j} = \mathcal{L}(h_{i,j}). $$
Since $h_{i,j}(\lambda)$ has degree $\nu(1)-1$, equality \eqref{eq:nu1:mmp} gives 
$$ {\omega}_{i,j}^{(\nu(1))} = \mathcal{L}^{(\nu(1))}(h_{i,j}) = \mathcal{L}(h_{i,j}) = \omega_{i,j}, $$
proving the theorem for $n=\nu(1)$.

Assume $n>\nu(1)$, with $t$ so that $\nu{(t)} \leq n < \nu(t+1)$,
and define the quadrature ${\mathcal{G}_n^{(n)}}$ for $\mathcal{L}^{(n)}$, determined by the polynomial $p_n(\lambda)$.
By \eqref{eq:cayley} and Theorem \ref{ade of wgq},  ${\mathcal{G}_n^{(n)}}(f) = \mathcal{L}^{(n)}(f)$ for every $f(\lambda) \in \mathcal{P}$.
Furthermore, $p_n(\lambda)$ determines the quadrature  $\mathcal{G}_n = \mathcal{G}_{\nu(t)}$ for $\mathcal{L}$, 
exact for every polynomials of degree at most $\nu(t) + \nu(t+1) - 2$.
As noticed above, $\mathcal{G}_n^{(n)}$ ad $\mathcal{G}_n$ coincide if and only if 
the respective weights $\omega_{i,j}^{(n)}$ and $\omega_{i,j}$ coincide.
Let $h_{i,j}(\lambda)$ be the interpolatory polynomials \eqref{eq:hij}. Since $h_{i,j}(\lambda)$ has degree $n-1$ the weight $\omega_{i,j}^{(n)}$ satisfies
$$ \omega_{i,j}^{(n)}  = \mathcal{L}^{(n)}(h_{i,j}) \, \mathbf{e}_1^T \, h_{i,j}(T_n) \, \mathbf{e}_{\nu{(1)}} 
= \mu \, m_{\nu{(1)} -1} \, \mathbf{e}_1^T \, h_{i,j}(T_{n-1}) \, \mathbf{e}_{\nu{(1)}} = \mathcal{L}(h_{i,j}) = \omega_{i,j},$$
where Lemma \ref{lemma:n-1} and the inductive assumption were used. 
\qed
\end{proof}

 We recall the definition of \emph{matrix function}.
A function  $f(\lambda)$ is defined on the spectrum of the given matrix $A$
when for every eigenvalue $\lambda_i$ of $A$ there exist $f^{(j)}(\lambda_i)$ for $j = 0,1, \dots, s_i - 1$, 
with $s_i$ the order of the largest Jordan block of $A$ in which $\lambda_i$ appears. 
Consider the Jordan block $\Lambda$ of the size $s$ corresponding to the eigenvalue $\lambda$,
then the matrix function $f(\Lambda)$ is defined as
$$   f(\Lambda) =  \left [ \begin{array}{ccccc}
                                      f(\lambda)      & \frac{f'(\lambda)}{1!} & \frac{f^{(2)}(\lambda)}{2!}   & \dots  & \frac{f^{(s -1)}(\lambda)}{(s -1)!}      \\
                                      0                 & f(\lambda)             & \frac{f'(\lambda)}{1!}        & \dots  & \frac{f^{(s -2)}(\lambda)}{(s -2)!}      \\
                                      \vdots            & \ddots                   & \ddots                          & \ddots & \vdots \\
                                      \vdots            &                          & \ddots                          & \ddots & \frac{f'(\lambda)}{1!} \\
                                      0                 & \dots                    & \dots                           & 0      & f(\lambda)
                                     \end{array} \right ].
$$

\noindent  Denoting
$$A = W \mathrm{diag}(\Lambda_1,\dots,\Lambda_\nu) W^{-1},$$
the Jordan decomposition of $A$, the matrix function $f(A)$ is defined as
\begin{equation}\label{eq:mtx:fun}
   f(A) = W \mathrm{diag} (f(\Lambda_1),\dots, f(\Lambda_\nu)) W^{-1}.
\end{equation}
We refer to \cite{HigBook08} for further information and for the equivalence to the other definitions of matrix function. 

Consider the block tridiagonal matrix $T_n$ of Theorem \ref{matching moments main}
and its Jordan decomposition $T_n = W \mathrm{diag} (\Lambda_1,\dots, \Lambda_\ell) W^{-1}$.
Since $T_n$ is nonderogatory, there are $\lambda_1, \dots, \lambda_\ell$ distinct eigenvalues
corresponding to the Jordan blocks $\Lambda_1, \dots, \Lambda_\ell$ of the sizes respectively $s_1, \dots, s_\ell$.
If $f(\lambda)$ is a smooth enough function so that $f(T_n)$ is well defined, 
then the Jordan decomposition of $T_n$ and some algebraic manipulations give
\begin{equation}\label{eq:mtx:form}
   \mu \, m_{\nu{(1)}-1} \, \mathbf{e}_1^T f(T_n) \, \mathbf{e}_{\nu{(1)}} = \sum_{i=1}^\ell \sum_{j=0}^{s_i-1} \omega_{i,j} f^{(j)}(\lambda_i) , 
\end{equation}
  with $\omega_{i,j}$ complex weights, $\mu$ and $m_{\nu{(1)}-1}$ as in Theorem \ref{matching moments main};
  see \cite{kautsky:81} and \cite[Section 3]{pinar:ramirez} for algebraic expressions of the weights.
  This observation together with the proof of Theorem \ref{matching moments main} 
  shows that when $n=\nu(t)$ the bilinear form $\mu \, m_{\nu{(1)}-1} \, \mathbf{e}_1^T f(T_n) \, \mathbf{e}_{\nu{(1)}}$
  is a matrix formulation of the $n$-node Gauss quadrature $\mathcal{G}_{n}(f)$ for the linear functional $\mathcal{L}$.
  Moreover, if $\nu(t) < n < \nu(t+1)$, then Lemma \ref{lemma:Gnreg} gives $\mathcal{G}_n = \mathcal{G}_{\nu(t)}$;
  hence $T_n$ and  $T_{\nu(t)}$ correspond to the same Gauss quadrature $\mathcal{G}_{\nu(t)}$, despite being different.
  
\section{The minimal partial realization and Gauss quadrature}\label{sec:minreal}
Any triplet $({\mathbf w}, A, {\mathbf v})$ composed of a matrix $A$ and vectors $\mathbf{v}, \mathbf{w}$, 
can be associated with a dynamical system
$$\frac{\text{d}{\mathbf z}}{\text{d}t}=A\,{\mathbf z}(t)+{\mathbf v} u(t)$$
$$y(t)={\mathbf w}^*{\mathbf z}(t),$$
with $\mathbf{z}(t)$ the state vector,  $u(t)$ the scalar input (control), and $y(t)$ the scalar output.
The transfer function 
$$\Gamma(\tau):={\mathbf w}^*(\tau I-A)^{-1}\,{\mathbf v}=\sum_{j=0}^{\infty}\frac{{\mathbf w}^* A^j\,{\mathbf v}}{\tau^{j+1}}$$
connects $u(t)$ with $y(t)$ and it is obtained applying the Laplace transform;
refer, e.g., to \cite[Section 2]{HoKal66}, \cite[Section 4]{Par92}, \cite[Section 4.1, 4.2 and 11.1]{AntBook05}. 
The series representation holds only for $|\tau|$ large enough, 
and the coefficients $\{{\mathbf w}^*A^j \, {\mathbf v}\}_{j=0}^{\infty}$ are usually known as \emph{Markov parameters}. 
The triplet $({\mathbf w}, A, {\mathbf v})$ is called a realization of $\Gamma$. 
One of the questions in systems theory is to determine all the realizations $({\mathbf w}, A, {\mathbf v})$ 
that yield a given (rational) function $\Gamma$, or equivalently, its Markov parameters. 
When the realization matches a finite number of Markov parameters it is said to be a \emph{partial realization}.
A partial realization in which $A$ has minimal dimension is called a \emph{minimal partial realization}.
Among the extensive literature about the realization problem we refer the reader to the papers by Kalman \cite{Kal63,Kal79},
Gilbert \cite{Gil63}, Ho and Kalman \cite{HoKal66}, Gragg \cite{Gra74}, Gragg and Lindquist \cite{GraLin83},
Parlett \cite{Par92} (which offers an algebraic point of view), Heinig and Jankowski \cite{HeiJan92}, 
and to the monographs by Kailath \cite{Kai80}, Bultheel and Van Barel \cite[Chapter 6]{BulVBaBook97}, 
Antoulas \cite[Section 4.4]{AntBook05}, and by Liesen and Strako\v s \cite[Section 3.9]{LieStrBook13}; 
see also \cite{Moo81}.
In the papers by Chebyshev from 1855--1859 \cite{Che1855,Che1859a} and Christoffel from 1858 \cite{Chr1858}
the concept equivalent to the  minimal partial realization is present (without using the name)
for a sequence of moments defining a positive definite linear functional;
cf. the comment in \cite[p.~23]{BulVBaBook97}.
The seminal paper by Stieltjes on continued fractions published in 1894 \cite[Sections 7--8, pp.~623--625, and Section 51, pp.~688--690]{Sti1894}
provides an instructive description;
see also \cite[Section 3.9.1]{LieStrBook13} and \cite{PozStr18}.
The results about the Gauss quadrature for real linear functionals 
and about the minimal partial realization of a sequence of real numbers appeared in the same year (1983)
respectively in the monograph by Draux \cite[Chapter 5]{DraBook83} and in the paper by Gragg and Lindquist \cite{GraLin83}.
Section \ref{sec:GQ} has presented the results by Draux extending them to the complex case.
Here the minimal partial realization of a sequence of complex numbers will be described 
together with the relationships between results in \cite{DraBook83} and \cite{GraLin83} 
(with extension to the complex case).

In the following we offer a non-standard formulation of the realization problem in systems theory.

\emph{Problem 1:} For a given finite sequence of complex numbers 
\begin{equation}\label{eq:seq:realiz}
   m_0, m_1, \dots, m_k,
\end{equation}
find all the triplets $({\mathbf w}, A, {\mathbf v})$ 
such that 
$$ {\mathbf w}^* A^j\,{\mathbf v} = m_j, \quad j=0,\dots,k. $$
Notice that usually the Markov parameters are defined as $\eta_j = m_{j-1}$. 

There always exists a solution of dimension $k+1$ of Problem 1.
For instance, take  $A\in \mathbb{C}^{k+1\times k+1}$ and $\mathbf{v}, \mathbf{w}\in \mathbb{C}^{k+1}$ as
\begin{equation}\label{eq:part:real}
  A = \left [\,\begin{matrix}
        \,\,  0 \,\, & \,\, 1 \,\,  &   &   \\
                     & 0            & \,\, \ddots \,\, &  \\
                     &              & \ddots           & \,\, 1 \,\, \\
                     &              &                  & 0
       \end{matrix} \, \right],
	\quad
	\mathbf{v} = \left [\,\begin{matrix}
	  m_0 \\  m_1   \\ \vdots  \\ m_{k}
       \end{matrix} \, \right],
	\quad
	\mathbf{w} = \left [\,\begin{matrix}
	  1 \\  0   \\ \vdots  \\ 0
       \end{matrix} \, \right].
\end{equation}
The sequence \eqref{eq:seq:realiz} defines the linear functional $\mathcal{L}$ on $\mathcal{P}_k$ 
with moments 
\begin{equation}\label{eq:linfun:realiz}
   \mathcal{L}(\lambda^j) = m_j, \quad j=0,\dots,k.
\end{equation}
For any solution $({\mathbf w}, A, {\mathbf v})$ of dimension $n$,
let $\lambda_1, \dots, \lambda_\ell$ be the distinct eigenvalues of $A$
and $s_i$ be the maximal geometric multiplicity of $\lambda_i$ (the size of the largest Jordan block corresponding to $\lambda_i$).
Then the definition of matrix function \eqref{eq:mtx:fun} and algebraic manipulations give
$$ {\mathbf w}^*f(A)\,{\mathbf v} = \sum_{i=1}^\ell \sum_{s=0}^{s_i-1} \omega_{i,s} f^{(s)}(\lambda_i), \quad s_1 + \dots + s_\ell \leq n,  $$
with $\omega_{i,s}$ complex weights. 
Therefore every realization of the sequence \eqref{eq:seq:realiz} defines a quadrature rule for the linear functional \eqref{eq:linfun:realiz}.

\emph{Problem 2:} Among all the realizations for \eqref{eq:seq:realiz} find those of smallest dimension.

Let $n$ be the smallest index so that
the unique $n$-node Gauss quadrature determined by the regular FOP $p_n(\lambda)$
is exact for every polynomial of degree smaller than or equal to $k$, i.e.,
$$ \mathcal{L}(q) = \mathcal{G}_n(q) =  \sum_{i=1}^\ell \sum_{s=0}^{s_i-1} \omega_{i,s} q^{(s)}(\lambda_i), \quad s_1 + \dots + s_\ell = n, \quad q \in \mathcal{P}_k. $$
If $T_n$ is the block tridiagonal matrix \eqref{block:Tn} corresponding to $p_n(\lambda)$,
then Theorem \ref{matching moments main} shows that the triplet 
$(\mathbf{e}_1, T_n, \mu \, m_{\nu{(1)}-1} \mathbf{e}_{\nu{(1)}})$
is a minimal partial realization for \eqref{eq:seq:realiz}.
All the other minimal partial realizations can be expressed as
\begin{equation}\label{eq:min:part:real}
    \left(B^*\mathbf{e}_1, \, B^{-1}T_n B, \, \mu \, m_{\nu{(1)}-1} B^{-1} \mathbf{e}_{\nu{(1)}}\right),
\end{equation}
with $B$ any $n\times n$ invertible matrix
(notice that this is a straightforward extension of the result given in \cite[Theorem 5]{GraLin83} to complex Markov parameters).
Hence any minimal partial realization of a sequence of complex number $m_0, m_1, \dots$ 
corresponds to a Gauss quadrature for the linear functional having $m_0, m_1, \dots$ as moments.

Finally, we recall the following well-known spectral result about minimal realizations, giving a proof based on the previous developments.
\begin{thm}\label{thm:spect:minreal}
Consider the matrix $A$ and the vectors $\mathbf{v}, \mathbf{w}$. 
   If the triplet $(\mathbf{c}, S, \mathbf{b})$ is a minimal realization of the sequence of Markov parameters given by
$$m_j = \mathbf{w}^* A^j \mathbf{v}, \quad j =0, 1, \dots ,$$
then the spectrum of $S$ is a subset of the spectrum of $A$.
\end{thm}
\begin{proof}
 Let $p_k(\lambda)$ be the characteristic polynomial of the matrix $A$ and consider the linear functional $\mathcal{L}$ defined by 
 $$ \mathcal{L}(q) = \mathbf{w}^* q(A) \, \mathbf{v}, \quad q(\lambda) \in \mathcal{P}. $$
 By Lemma \ref{thm:char:poly} and the Cayley--Hamilton Theorem 
 the $k$-degree polynomial $p_k(\lambda)$ is formally orthogonal to every polynomial, i.e., 
 $\mathcal{L}(p_k q) = 0$ for every  $q(\lambda) \in \mathcal{P}$.
 Consider the last regular FOP $p_n(\lambda)$ in the sequence of the FOPs with respect to $\mathcal{L}$.
 The polynomial $p_k(\lambda)$ is $(k-n)$-quasi-orthogonal (note that $n \leq k$). 
 Hence the roots of $p_n(\lambda)$ are roots of $p_k(\lambda)$ by Theorem \ref{theorem rr for mqop}.
 As discussed above, every minimal realization can be expressed as
 $\left(B^*\mathbf{e}_1, \, B^{-1}T_n B, \, \mu \, m_{\nu{(1)}-1} B^{-1} \mathbf{e}_{\nu{(1)}}\right)$,
 with $T_n$ the block tridiagonal matrix \eqref{block:Tn} corresponding to $p_n(\lambda)$
 and $B$ an invertible matrix.
 Thus Lemma \ref{thm:char:poly} concludes the proof.
 \qed
\end{proof}

\noindent We remark that the previous theorem is a consequence of the Canonical Structure Theorem of the linear system theory;
see, e.g., \cite{Kal62}, \cite[Theorem 5]{Kal63}, \cite{Gil63} and the description in \cite[Section 7]{Par92}.

\section{The look-ahead Lanczos algorithm and Gauss quadrature}\label{sec:lanczos}
   Consider a complex matrix $A$  and a complex vector $\mathbf{v}$ of the corresponding dimension.
   The $n$th \emph{Krylov subspace} generated by $A$ and $\mathbf{v}$ 
   is the subspace 
   $$ \mathcal{K}_n(A,\mathbf{v}) =   
	    \textrm{span}\{\mathbf{v},A\, \mathbf{v},\dots,A^{n-1} \,\mathbf{v}\},$$
   which can be equivalently expressed as
   $$ \mathcal{K}_n(A,\mathbf{v}) = \{ p(A)\, \mathbf{v} : p(\lambda) \in \mathcal{P}_{n-1}\}. $$ 
   The basic facts about Krylov subspaces had been given by Gantmacher in \cite{Gan34};
   other results can be found, e.g., in \cite[Section 2.2]{LieStrBook13}.
   
   Let $A$ be a complex matrix, $\mathbf{v}, \mathbf{w}$ be complex vectors,
   and $\mathcal{L}: \mathcal{P} \rightarrow \mathbb{C}$ be the linear functional defined by
   \begin{equation}\label{funct nonH}
      \mathcal{L}(p) = \mathbf{w}^* p(A) \, \mathbf{v}, \quad p(\lambda) \in \mathcal{P}.
   \end{equation}  
  Denoting with $\bar p(\lambda)$ the polynomial whose coefficients are the conjugates of the coefficients of $p(\lambda)$
  and noticing that  
   $$ p(A)^* = \bar p(A^*),$$
   for $p(\lambda), q(\lambda) \in \mathcal{P}_{n-1}$, give
   $$ \mathcal{L}(qp)  = \mathbf{w}^*q(A)p(A)\,\mathbf{v} =  \mathbf{\widehat w}^* \mathbf{\widehat v},$$
   with $\mathbf{\widehat v} = p(A) \, \mathbf{v} \in \mathcal{K}_n(A,\mathbf{v})$
   and $\mathbf{\widehat w} = \bar q(A^*) \, \mathbf{w} \in \mathcal{K}_n(A^*,\mathbf{w})$.
   
   The non-Hermitian Lanczos algorithm (formulated by Lanczos in \cite{Lan50} and \cite{Lan52})
   gives, when possible, the vectors 
   $$ \mathbf{v}_0, \dots, \mathbf{v}_{n-1} \quad \textrm{ and } \quad \mathbf{w}_0, \dots, \mathbf{w}_{n-1}, $$
   which are respectively basis of $\mathcal{K}_n(A,\mathbf{v})$ and $\mathcal{K}_n(A^*,\mathbf{w})$ 
   satisfying the biorthogonality conditions
   \begin{equation}\label{eq:orth:cond:lanczos}
    \mathbf{w}_i^*\mathbf{v}_j = 0, \quad i \neq j, \quad \textrm{ and } \quad 
	\mathbf{w}_i^*\mathbf{v}_i \neq 0, \quad i,j=0,\dots,n-1. 
   \end{equation}
   In this case, there exist regular FOPs $p_0(\lambda), \dots, p_{n-1}(\lambda)$
   with respect to the linear functional \eqref{funct nonH} so that
   $$ \mathbf{v}_j = p_j(A)\,\mathbf{v} \quad \textrm{ and } \quad \mathbf{w}_j = \bar p_j(A^*)\,\mathbf{w}, \quad j=0,\dots,n-1.$$
   Hence bases satisfying \eqref{eq:orth:cond:lanczos} exist if and only if $\mathcal{L}$ is quasi-definite on $\mathcal{P}_{n-1}$;
   see, e.g., \cite[Theorem 2.1]{PozPraStr18}.
   
   In the non-Hermitian Lanczos algorithm, the vectors $\mathbf{v}_j, \mathbf{w}_j$, $j=0,\dots,n-1$,
   are obtained by the three-term recurrences satisfied by the regular FOPs $p_0, \dots, p_{n-1}$;
   for details refer to \cite[Section 2.7.2]{BreBook80}, \cite{Gut92,Gut94b,Gut94}, 
   \cite[Chapter 7]{SaaBook03}, \cite[Chapter 4]{GolMeuBook10}, \cite[Section 2.4]{LieStrBook13},
   also refer to the survey \cite{PozPraStr18} 
   where the connection with the Gauss quadrature for quasi-definite linear functionals is described.     
   Considering biorthonormal vectors, i.e., $\mathbf{w}_i^*\mathbf{v}_i = 1$,
   the non-Hermitian Lanczos algorithm corresponds to the three-term recurrences \eqref{eq:3term:ast1}
   and can be given as Algorithm \ref{algo non-H Lanczos};
   see, e.g., \cite{Cull86,CULLUM198919}.
   The outputs of the first $n-1$ iterations of Algorithm \ref{algo non-H Lanczos}
   define the matrices
   $$V_n = [\mathbf{v}_0, \dots, \mathbf{v}_{n-1}] \quad \text{ and } \quad W_n = [\mathbf{w}_0, \dots, \mathbf{w}_{n-1}]$$
   which satisfy $W_n^{*} V_n = I$, with $I$ the identity matrix of dimension $n$.
   Moreover, 
  \begin{align*}
   A V_n = V_n J_n + \widehat{\mathbf{v}}_n \mathbf{e}_n^T, \\
   A^* W_n = W_n \bar{J}_n + \widehat{\mathbf{w}}_n \mathbf{e}_n^T,
  \end{align*}
   with $J_n$ the complex Jacobi matrix \eqref{eq:jacobi} associated with the linear functional \eqref{funct nonH},
   and $\bar{J}_n$ the Jacobi matrix with conjugate elements ($\widehat{\mathbf{v}}_n$ and $\widehat{\mathbf{w}}_n$ are defined in Algorithm \ref{algo non-H Lanczos}).
Therefore the non-Hermitian Lanczos algorithm can be seen as a way to compute $J_n$ and hence the Gauss quadrature for the functional \eqref{funct nonH}; 
see \cite[Theorem 2]{FreHoc93} and also \cite{PozPraStr18}
(for the block Lanczos algorithm see, e.g., \cite[Section 3]{fenu:reichel:2013}).

\begin{table}
     \noindent\fbox{
\parbox{0.95\textwidth}{\begin{center}
\parbox{0.9\textwidth}{
\begin{algorithm}[non-Hermitian Lanczos algorithm]\label{algo non-H Lanczos} $ $

 \noindent Input: a complex matrix $A$, and
                  complex vectors $\mathbf{v},\mathbf{w}$ such that $\mathbf{w}^*\mathbf{v} \neq 0$.

 \noindent Output: vectors $\mathbf{v}_0,\dots,\mathbf{v}_{n-1}$ and vectors $\mathbf{w}_0,\dots,\mathbf{w}_{n-1}$
  spanning respectively $\mathcal{K}_n(A,\mathbf{v})$, $\mathcal{K}_n(A^*,\mathbf{w})$
  and satisfying the biorthogonality conditions \eqref{eq:orth:cond:lanczos} with $\mathbf{w}_i^*\mathbf{v}_i = 1$, $i=0,\dots, n-1$.
  \begin{align*}
  & \textrm{ Initialize: } \mathbf{v}_{-1}=\mathbf{w}_{-1}=0, \, \beta_0 =\sqrt{\mathbf{w}^*\mathbf{v}}, \, \mathbf{v}_0 = \mathbf{v}/\beta_0, \, \mathbf{w}_0 = \mathbf{w}/ \bar{\beta}_0. \\
  & \textrm{ For } n=1,2,\dots \\ 
  &    \qquad \quad \alpha_{n-1} = \mathbf{w}_{n-1}^*A\mathbf{v}_{n-1}, \\
  &    \qquad \quad \mathbf{\widehat v}_{n} = A \mathbf{v}_{n-1} - \alpha_{n-1}\mathbf{v}_{n-1} - \beta_{n-1} \mathbf{v}_{n-2}, \\
  &    \qquad \quad \mathbf{\widehat w}_{n} = A^*\mathbf{w}_{n-1} - \bar{\alpha}_{n-1}\mathbf{w}_{n-1} - \bar \beta_{n-1} \mathbf{w}_{n-2}, \\
  &     \qquad \quad \beta_n = \sqrt{\mathbf{\widehat w}_{n}^*\mathbf{\widehat v}_n}, \\
  &     \qquad \quad \textrm{if } \beta_n = 0 \textrm{ then stop}, \\
  &     \qquad \quad  \mathbf{v}_n = \mathbf{\widehat v}_n / \beta_n, \\
  &     \qquad \quad  \mathbf{w}_n = \mathbf{\widehat w}_n / \bar \beta_n, \\
  & end.
  \end{align*}
\end{algorithm}
  }\end{center}}}
  \end{table}

   If the $n$th iteration of Algorithm \ref{algo non-H Lanczos} gives $\beta_n = 0$, then the algorithm has a \emph{breakdown}.
   Since $\beta_n = \mathcal{L}(\lambda  p_{n-1} p_{n})$,
   a breakdown arises if and only if $\mathcal{L}$ is not quasi-definite on $\mathcal{P}_n$. 
   In this case, the FOP $p_{n}(\lambda)$ is orthogonal to itself.
   Therefore there do not exist biorthonormal bases of the Krylov subspaces 
   $\mathcal{K}_{n+1}(A,\mathbf{v})$ and $\mathcal{K}_{n+1}(A^*,\mathbf{w})$.
   Moreover, there does not exist a regular FOP $p_{n+1}(\lambda)$.
  There are two kinds of breakdown for Algorithm \ref{algo non-H Lanczos}:
  \begin{enumerate}
  \item \emph{lucky breakdown}  (or \emph{benign breakdown}), when $\widehat{\mathbf{v}}_n = 0$ or $\widehat{\mathbf{w}}_n = 0$;
  \item \emph{serious breakdown}, when $\widehat{\mathbf{v}}_n\neq \mathbf{0}$ and
	  $\widehat{\mathbf{w}}_n\neq\mathbf{0}$, 
	  but $\mathbf{\widehat w}_{n}^*\mathbf{\widehat v}_n=0$.
  \end{enumerate}
In the first case either $\mathcal{K}_{n}(A,\mathbf{v})$ is $A$-invariant or $\mathcal{K}_{n}(A^*,\mathbf{w})$ is $A^*$-invariant.
Then the algorithm is usually stopped since an invariant subspace is often a desirable result; 
see, e.g., \cite{BreRedSad91}, \cite[Section 5]{Par92} and \cite[Section 10.5.5]{GolVLoBook13}.
The second case is problematic. 
In \cite[pp.~389--391]{WilBook65} 
Wilkinson showed with some examples that well-conditioned matrices with well-conditioned eigenvectors can produce a breakdown.
Hence as Wilkinson wrote, serious breakdown ``is not associated with any shortcoming in the matrix $A$. 
It can happen even when the eigenproblem of $A$ is very well-conditioned. 
We are forced to regard it as a specific weakness of the Lanczos method itself.''
The interested reader can also refer to 
\cite{Rut53}, \cite[p.~34]{HouBau59}, \cite[Chapter IV]{Tay82}, \cite{ParTayLiu85},
  \cite[Section 7]{Par92}, and \cite{Gut92,Gut94b,Gut94}.

Taylor in \cite{Tay82} and Parlett, Taylor, and Liu in \cite{ParTayLiu85}
first proposed the \emph{look-ahead Lanczos algorithm},
a strategy able to deal with the breakdown problem.
When $\mathbf{\widehat w}_{n}^*\mathbf{\widehat v}_n=0$, 
the idea behind their strategy is to look for a vector $\mathbf{\widetilde w}_{k} \in \mathcal{K}_{k+1}(A^*, \mathbf{w})$,
with $k>n$ big enough, so that $\mathbf{\widetilde w}_{k}^*\mathbf{\widehat v}_n\neq0$ 
and $\mathbf{\widetilde w}_{k}^*\mathbf{v}_j = 0$ for $j=0,\dots,n-1$.
In \cite{FreGutNac93} Freund, Gutknecht, and Nachtigal implemented a different look-ahead strategy
considering sequences of FOPs and quasi-orthogonal polynomials.
Their procedure is based on the work of Gutknecht published in \cite{Gut92} and later in \cite{Gut94b};
see also the thesis \cite{Nac91} by Nachtigal and the description in \cite{Fre93b} by Freund.
We also refer the reader to the strategy in \cite{BreRedSad91,Brezinski1992} and the related work \cite{Draux96}. 
The following part will describe the basic ideas behind the look-ahead Lanczos algorithm by Freund, Gutknecht, and Nachtigal.

Consider the linear functional \eqref{funct nonH} and let $p_0(\lambda)\neq 0,p_1(\lambda),\dots$ 
be the sequence \eqref{eq:seq:p} of polynomials so that 
$p_{\nu(0)}(\lambda) = p_0(\lambda), p_{\nu(1)}(\lambda),  \dots $ are the regular FOPs 
and $p_n$ is an $(n-\nu(t))$-quasi-orthogonal polynomial for $\nu(t) < n < \nu(t+1)$,
with $\nu(t+1) = \infty$ when $p_{\nu(t)}$ is the last of the regular FOPs.
Moreover, consider the vectors
\begin{equation*}
 \mathbf{v}_n = p_n(A)\,\mathbf{v} \quad \textrm{ and } \quad \mathbf{w}_n = \bar p_n(A^*)\,\mathbf{w}, \quad n=0, 1, \dots ,
\end{equation*}
and the matrices $V^{(t)}_n = [\mathbf{v}_{\nu(t)}, \dots, \mathbf{v}_{n-1}]$,
$W^{(t)}_n =[\mathbf{w}_{\nu(t)}, \dots, \mathbf{w}_{n-1}]$,
with $V^{(t)} = V^{(t)}_{\nu(t+1)}$ and $W^{(t)} = W^{(t)}_{\nu(t+1)}$
for simplicity of notation.
Hence for $\nu(t) < n \leq \nu(t+1)$, the columns of $V_n = [V^{(0)}, \dots, V^{(t)}_n]$ and of $W_n = [W^{(0)}, \dots, W^{(t)}_n]$
are respectively basis of $\mathcal{K}_n(A, \mathbf{v})$ and $\mathcal{K}_n(A^*, \mathbf{w})$.
However, instead of the biorthogonality conditions \eqref{eq:orth:cond:lanczos}, the following block-biorthogonality conditions hold
\begin{equation}\label{eq:block:biorth}
   W^{(t)*}_n V^{(k)}_\ell = 0, \, \textrm{ for } t \neq k, 
\quad \textrm{ and } \quad 
   W^{(t)*}_n V^{(t)}_n = \Omega^{(t)}_n,
\end{equation}
 with
 $$
 \Omega^{(t)}_n = \left[ \begin{matrix}
                        \mathcal{L}(p_{\nu(t)}, p_{\nu(t)}) & \dots & \mathcal{L}(p_{\nu(t)}, p_{n-1}) \\
                        \vdots &   & \vdots \\
                        \mathcal{L}(p_{n-1}, p_{\nu(t)}) 	  & \dots & \mathcal{L}(p_{n-1}, p_{n-1}) \\
                     \end{matrix} \right];
 $$
 we denote $\Omega^{(t)}_{\nu(t+1)}$ by $\Omega^{(t)}$. 
 
 By Theorem \ref{theorem rr for mqop} and the recurrences \eqref{eq:gen:rec:qorth}
 if $\nu(t) < n < \nu(t+1)$, 
 then for some complex coefficients $\mathbf{a}_n = [\alpha_{n,\nu(t)}, \dots, \alpha_{n,n-1}]$ and $\beta_n \neq 0$ 
 the following recurrences hold
 $$ \beta_n\mathbf{v}_n = A \mathbf{v}_{n-1} - \sum_{j=\nu(t)}^{n-1} \alpha_{n,j} \mathbf{v}_j,
    \quad \textrm{ and } \quad
    \bar \beta_n\mathbf{w}_n = A^* \mathbf{w}_{n-1} - \sum_{j=\nu(t)}^{n-1} \bar \alpha_{n,j} \mathbf{w}_j.$$
If $n = \nu(t+1)$ with $t \geq 0$, then for some $\beta_n \neq 0$ the recurrences \eqref{long rr by Freund} give
\begin{align*}
 \beta_n \mathbf{v}_n &= A\mathbf{v}_{n-1}           - \sum_{j=\nu(t)}^{n-1}\alpha_{n,j} \mathbf{v}_j - \gamma_{n}\mathbf{v}_{\nu(t-1)} \\
   \bar \beta_n \mathbf{w}_n &= A^*\mathbf{w}_{n-1}  - \sum_{j=\nu(t)}^{n-1} \bar \alpha_{n,j} \mathbf{w}_j - \bar \gamma_{n}\mathbf{w}_{\nu(t-1)},   
\end{align*}
where $\nu(-1)=-1$, $\mathbf{v}_{-1} = \mathbf{w}_{-1} = 0$, $\gamma_{\nu(1)} = 0$,
$$ \gamma_n = \frac{ \mathbf{w}_{\nu(t)-1}^* A \, \mathbf{v}_{n-1}}{\mathbf{w}_{\nu(t)-1}^* \mathbf{v}_{\nu(t-1)}}, \quad t \geq 1, $$
and the coefficients $\mathbf{a}_n = [\alpha_{n,\nu(t)}, \dots, \alpha_{n,n-1}]$ are given as the solution of the system 
$$ \Omega^{(t)} \, \mathbf{a}_n = W^{(t)*} A \mathbf{v}_{n-1};  $$
see the linear system \eqref{system for alpha}.
The described recurrences can be expressed in the matrix form
$$ A V_n = V_n T_n^T + \beta_{n+1}\mathbf{v}_{n}\mathbf{e}_{n}^T
\quad \textrm{ and } \quad
  A^* W_n = W_n T_n^* + \bar \beta_{n+1} \mathbf{w}_{n}\mathbf{e}_{n}^T, $$
with $T_n^T$ the transpose of the block tridiagonal matrix $T_n$ defined in \eqref{block:Tn}
and $T_n^*$ the conjugate transpose of $T_n$.
The resulting form of the look-ahead Lanczos algorithm is given as Algorithm \ref{algo:look:Lanczos}
and corresponds to the algorithm proposed in \cite[Algorithm 3.1]{FreGutNac93}; see also \cite[Algorithm 5.1]{Fre93b}.

\begin{table}[th!]
     \noindent\fbox{
\parbox{0.95\textwidth}{\begin{center}
\parbox{0.9\textwidth}{
\begin{algorithm}[look-ahead Lanczos algorithm]\label{algo:look:Lanczos} $ $

 \noindent Input: a complex matrix $A$ and complex vectors $\mathbf{v},\mathbf{w} \neq 0$.

 \noindent Output: vectors $\mathbf{v}_0,\dots,\mathbf{v}_{n-1}$ and vectors $\mathbf{w}_0,\dots,\mathbf{w}_{n-1}$
  spanning respectively $\mathcal{K}_n(A,\mathbf{v})$, $\mathcal{K}_n(A^*,\mathbf{w})$ 
  and satisfying the block biorthogonality conditions \eqref{eq:block:biorth}.
  \begin{align*}
  & \textrm{ Initialize: } \nu(-1)=-1, \, \mathbf{v}_{-1}=\mathbf{w}_{-1}=0, \, t = \nu(0)=0,  \, \eta_0 = 1, \textrm{ fix } \beta_0 \neq 0, \\
  & \qquad \qquad \quad    \mathbf{v}_0 = \mathbf{v}/\beta_0, \, \mathbf{w}_0 = \mathbf{w}/ \bar \beta_0, \, V_1^{(0)} = [\mathbf{v}_0], \, W_1^{(0)} = [\mathbf{w}_0]. \\ 
  & \textrm{ For } n=1,2,\dots, \\
  & \qquad \Omega_n^{(t)} = W^{(t)*}_n V^{(t)}_n, \\
  & \qquad \textrm{If $\Omega_{n}^{(t)}$ is regular, then } \\
  &    \qquad\qquad \quad \gamma_n = {(\mathbf{w}_{\nu(t)-1}^* A \mathbf{v}_{n-1})}/ \eta_t, \\
  &    \qquad\qquad \quad \mathbf{a}_n = \left(\Omega_{n}^{(t)}\right)^{-1} W_{n}^{(t)} A \mathbf{v}_{n-1}, \\
  &    \qquad\qquad \quad \mathbf{\widehat v}_{n} = A \mathbf{v}_{n-1} - V_{n}^{(t)} \mathbf{a}_n - \gamma_{n} \mathbf{v}_{\nu(t-1)}, \\
  &    \qquad\qquad \quad \mathbf{\widehat w}_{n} = A^*\mathbf{w}_{n-1} - W_{n}^{(t)} \bar{\mathbf{a}}_n - \bar \gamma_{n} \mathbf{w}_{\nu(t-1)}, \\
  &    \qquad\qquad \quad \textrm{fix } \beta_n \neq 0, \\
  &    \qquad\qquad \quad  \mathbf{v}_n = \mathbf{\widehat v}_n / \beta_n,\,  \mathbf{w}_n = \mathbf{\widehat w}_n / \bar \beta_n, \\
  &    \qquad\qquad \quad \eta_{t+1} = (\mathbf{w}_{n - 1}^* \mathbf{v}_{\nu(t)}), \\
  &    \qquad\qquad \quad  t = t + 1, \, \nu(t) = n, \, V_{n+1}^{(t)} = [\mathbf{v}_n] , \, W_{n+1}^{(t)} = [\mathbf{w}_n], \\
  & \qquad \textrm{else if $\Omega_{n}^{(t)}$ is singular, then } \\
  &    \qquad\qquad \quad \textrm{fix the vector } \mathbf{a}_n, \\
  &    \qquad\qquad \quad \mathbf{\widehat v}_{n} = A \mathbf{v}_{n-1} - V_{n}^{(t)} \mathbf{a}_n, \\
  &    \qquad\qquad \quad \mathbf{\widehat w}_{n} = A^*\mathbf{w}_{n-1} - W_{n}^{(t)} \bar{\mathbf{a}}_n, \\
  &    \qquad\qquad \quad \textrm{fix } \beta_n \neq 0, \\
  &    \qquad\qquad \quad  \mathbf{v}_n = \mathbf{\widehat v}_n / \beta_n, \, \mathbf{w}_n = \mathbf{\widehat w}_n / \bar \beta_n, \\
  &    \qquad\qquad \quad V_{n+1}^{(t)} = [V_{n}^{(t)}, \mathbf{v}_n], \, W_{n+1}^{(t)} = [W_{n}^{(t)}, \mathbf{w}_n], \\
  & \qquad\textrm{end if} \\
  & \qquad \textrm{If }\, \mathbf{v}_n = 0 \textrm{ or }\, \mathbf{w}_n = 0 \textrm{ then } \\
  &    \qquad\qquad \quad \textrm{stop}, \\
  & \qquad\textrm{end if} \\
  & \textrm{end for}.
  \end{align*}
\end{algorithm}
  }\end{center}}}
  \end{table}

  The first $n$ iterations of Algorithm \ref{algo:look:Lanczos}
  produce the coefficients of the block tridiagonal matrix $T_n$. 
  If $\nu(t) \leq n < \nu(t+1)$,
  then the Gauss quadrature $\mathcal{G}_{\nu(t)}$ for the linear functional \eqref{funct nonH}
  has the matrix formulation \eqref{eq:mtx:form} which is determined by the matrix $T_n$.
  Hence Algorithm \ref{algo:look:Lanczos} produces Gauss quadratures for the linear functional \eqref{funct nonH}.
  Notice that by Lemma \ref{lemma:Gnreg}  
  the matrix $T_n$ corresponds to the Gauss quadrature $\mathcal{G}_{\nu(t)}$ for $n = \nu(t), \dots, \nu(t+1)-1$.
  Nevertheless, the iterations $\nu(t)+1, \dots, \nu(t+1)$ of Algorithm \ref{algo:look:Lanczos} are informative
  since they show that $\mathcal{G}_{\nu(t)}$ has degree of exactness larger than $2\nu(t)-1$.
  At the same time, Algorithm \ref{algo:look:Lanczos} also produces the triplet $(\mathbf{e}_1, T_{\nu(t)}, \mu \, m_{\nu{(1)}-1} \mathbf{e}_{\nu{(1)}})$, 
  i.e., the minimal partial realization \eqref{eq:min:part:real} (with $B = I$) of the sequence of Markov parameters defined by
  \begin{equation}\label{eq:markov:par}
     m_j = \mathbf{w}^* A^j \, \mathbf{v}, \quad \textrm{ for } j=0,1,\dots \,.
  \end{equation}

  Consider the case in which a benign breakdown does not arise
  and the determinants of the Hankel submatrices \eqref{eq:hankel:subm} composed of the moments 
  of the linear functional \eqref{funct nonH} are such that
\begin{equation}\label{incurable breakdown}
    \Delta_{n-1}\neq 0,\quad \Delta_{n+k} = 0, \quad \textrm{ for } k = 0, 1, \dots,
\end{equation}
known as \emph{incurable breakdown};
see \cite[page 56]{Tay82}, \cite[Section 7]{ParTayLiu85}, \cite[p.~577]{Par92}.
Then $p_n(\lambda)$ is the last of the regular FOPs ($n = \nu(t)$).
By Theorem \ref{ade of wgq},  
the quadrature $\mathcal{G}_n$ determined by $p_n(\lambda)$ is the Gauss quadrature with maximal number of nodes (counting the multiplicities)
and it is exact for every polynomial.
Equivalently, let $T_n$ be the block tridiagonal matrix obtained at the $n$th step of the Lanczos algorithm. 
Theorem \ref{matching moments main} gives
$$  \mathcal{L}(p) = \mathbf{w}^* p(A)\, \mathbf{v} = \mathcal{G}_n(p) = \mu \, m_{\nu{(1)}-1} \, \mathbf{e}_1^T p(T_n) \, \mathbf{e}_{\nu{(1)}}, \quad p(\lambda) \in \mathcal{P}. $$
Moreover, if $f(\lambda)$ is a function so that $f(A)$ and $f(T_n)$ are well defined matrix functions,
then there exists a polynomial $q(\lambda)$ interpolating in the Hermite sense $f(\lambda)$ at the spectra of $A$ and $T_n$ (note that $q(\lambda)$ depends on $f(\lambda)$, $A$, and $T_m$);
see, e.g., \cite[Section 1.2]{HigBook08}. 
Therefore 
$$\mathcal{L}(f) = \mathbf{w}^* q(A) \, \mathbf{v} = \mu \, m_{\nu{(1)}-1} \, \mathbf{e}_1^T q(T_n) \, \mathbf{e}_{\nu{(1)}} = \mathcal{G}_n(f).$$
Looking at the Lanczos algorithm as a method for getting the Gauss quadrature for a linear functional \eqref{funct nonH},
the incurable breakdown corresponds to the solution of the problem as well as the lucky breakdown.
Furthermore, the triplet $(\mathbf{e}_1, T_n, \mu \, m_{\nu{(1)}-1} \mathbf{e}_{\nu{(1)}})$
is a minimal realization of the transfer function associated with $({\mathbf w}, A, {\mathbf v})$,
i.e., it matches the Markov parameters \eqref{eq:markov:par}.
The previous considerations together with Theorem \ref{thm:spect:minreal} give a new proof for the Mismatch Theorem
based on the properties of the Gauss quadrature for linear functionals. 
The Mismatch Theorem was first proved in \cite[Theorem 4.2]{Tay82} by Taylor; see also \cite[p.~117]{ParTayLiu85}, 
and \cite[Section 7]{Par92} where the theorem was connected with the minimal realization problem.
\begin{thm}[Mismatch Theorem]\label{thm:taylor}
   Le $T_n$ be the block tridiagonal matrix obtained at the $n$th step of Algorithm \ref{algo:look:Lanczos}
   with $A$ as the input matrix and $\mathbf{w},\mathbf{v}\neq 0$ as the input vectors. 
   If the algorithm has an incurable breakdown at the $n$th step, i.e., 
   the Hankel determinants corresponding to the linear functional \eqref{funct nonH}  satisfy \eqref{incurable breakdown},
   then each eigenvalue of $T_n$ (known as Ritz value) is an eigenvalue of $A$.
\end{thm}

\noindent Notice that the look-ahead Lanczos algorithm in \cite{Tay82}
 produces a block tridiagonal matrix different from the matrix $T_n$ \eqref{block:Tn}.
 However, both the matrices are minimal realization of the same sequence of numbers and therefore they are similar.

\section{Conclusion}\label{sec:conc}
The $n$-node Gauss quadrature $\mathcal{G}_n$ for a linear functional $\mathcal{L}$ described in Section \ref{sec:GQ} 
is a straightforward extension of the quadrature introduced for real-valued linear functionals in \cite[Chapter 5]{DraBook83} to the complex case 
and it satisfies the following properties:
\begin{enumerate}
  \item the Gauss quadrature $\mathcal{G}_n$ has degree of exactness \emph{at least} $2n - 1$;
  \item the Gauss quadrature $\mathcal{G}_n$ exists and is unique if and only if the Hankel submatrix of moments $H_{n-1}$ is nonsingular, i.e., $\Delta_{n-1}\neq 0$;
  \item by Theorem \ref{matching moments main} the Gauss quadrature can be written in the matrix form 
	    $ \mathcal{G}_n(f) = \mu \, m_{\nu{(1)}-1} \, \mathbf{e}_1^T f(T_n) \, \mathbf{e}_{\nu{(1)}}$.
 \end{enumerate}
Note that such properties are weaker forms of the properties G1--G3 in Section \ref{sec:qdef}.

  \begin{figure}[tb]
\centering
\includegraphics[width = 0.9\textwidth]{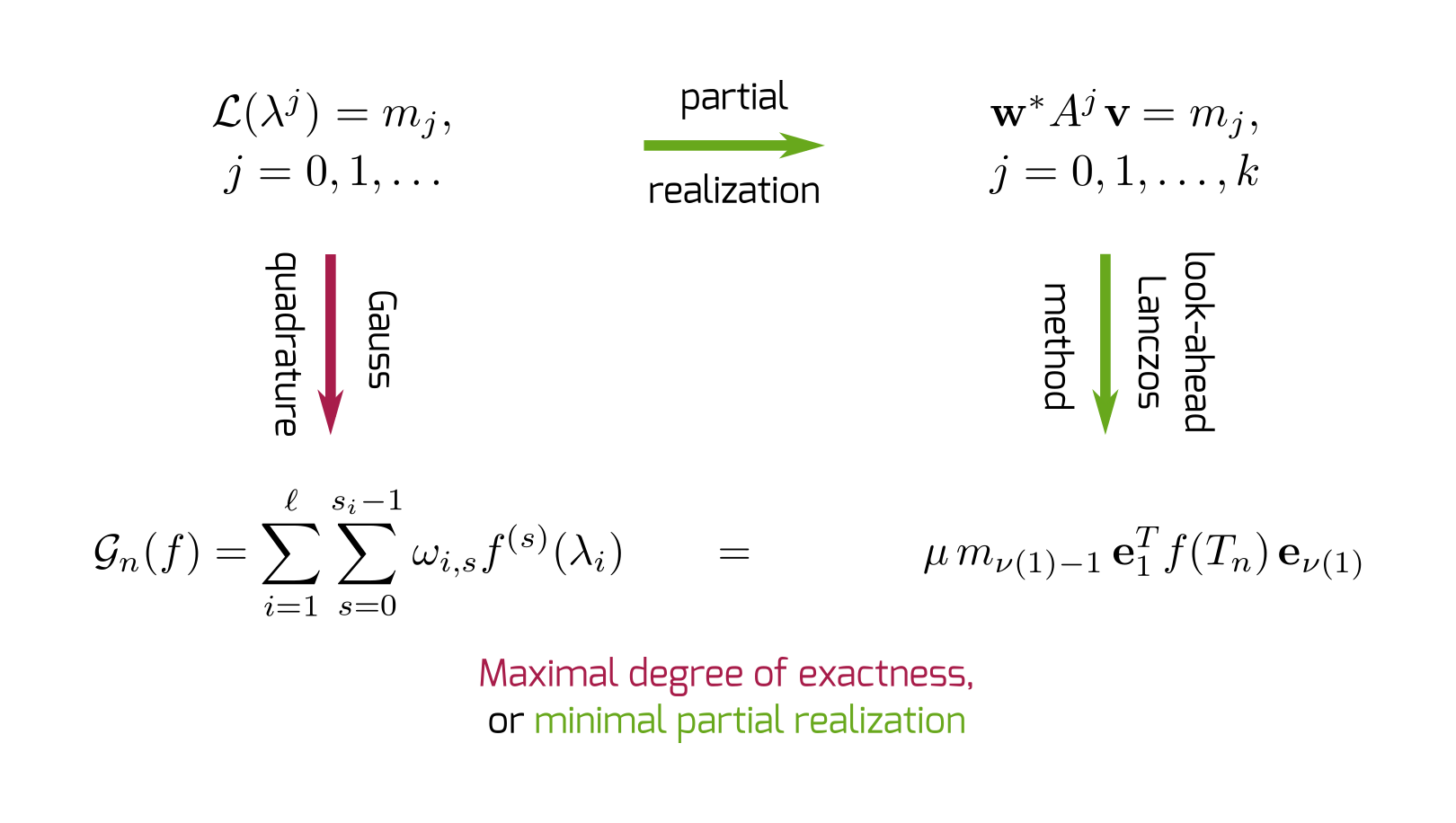}
\caption{Visualization of the connections between the Gauss quadrature for linear functionals, minimal partial realization, and look-ahead Lanczos algorithm.}
\label{fig:scheme}
\end{figure}

Figure \ref{fig:scheme} summarizes the connections
between the Gauss quadrature for linear functionals, minimal partial realization, and look-ahead Lanczos algorithm.
On the right-hand side, 
the triplet $({\mathbf w}, A, {\mathbf v})$ is a partial realization matching the first $k+1$ elements of the sequence of complex numbers $m_0, m_1, \dots$ .
A minimal partial realization
can be obtained applying the look-ahead Lanczos algorithm 
to the matrix $A$ and the vectors $\mathbf{v}, \mathbf{w}$
(this is also connected with the concept of model reduction, see, e.g., \cite[Chapter 3, in particular Section 3.9]{LieStrBook13}).
Notice that
the Lanczos algorithm applied to the partial realization \eqref{eq:part:real} 
is related to the Berlekamp-Massey algorithm \cite{Ber68,Mass69} (see \cite{Kun77}, \cite{GraLin83}, and \cite{BoLeLu92}).
On the left-hand side, 
the sequence $m_0, m_1, \dots$ determines the linear functional $\mathcal{L}: \mathcal{P} \rightarrow \mathbb{C}$ 
by defining its moments. 
The functional $\mathcal{L}$ can be approximated by a Gauss quadrature. 
Among all the Gauss quadratures exact on $\mathcal{P}_k$,
there is one with the minimal number of nodes $n$ (counting the multiplicities).
Such quadrature can be written in the matrix form
$$ \mathcal{G}_n(f) = \mu \, m_{\nu{(1)}-1} \, \mathbf{e}_1^T f(T_n) \, \mathbf{e}_{\nu{(1)}}, $$
i.e., it corresponds to the minimal partial realization matching $m_0, \dots, m_k$.

In Sections \ref{sec:minreal} and \ref{sec:lanczos}, we discussed the correspondence between the incurable breakdown in the look-ahead Lanczos algorithm and the minimal realization of an infinite sequence of complex numbers (and to the unique Gauss quadrature exact for every polynomial).
This connection led us to a new proof for the Mismatch Theorem \ref{thm:taylor}.

\begin{acknowledgements}
We would like to thank Zden\v ek Strako\v s for the helpful comments and improvements suggested.
This work has been supported by Charles University Research program No. UNCE/SCI/023 
and by the Ministry for Scientific and Technological Development, Higher Education and Information Society of R. Srpska.
\end{acknowledgements}

%
%

\bibliographystyle{spmpsci}
\bibliography{mybib}



\end{document}